\documentclass[12pt, a4paper]{amsart}
\usepackage{amsmath,amsfonts,amssymb,amsthm, color, mathrsfs}
\usepackage{hyperref}
\newtheorem{thm}{Theorem}[section]
\newtheorem{prop}[thm]{Proposition}
\newtheorem{coro}[thm]{Corollary}

\newtheorem{lem}[thm]{Lemma}
\newtheorem{rem}[thm]{Remark}

\theoremstyle{definition}
\newtheorem{de}[thm]{Definition}

\newcommand{\C}{{\mathbb{C}}}
\newcommand{\Z}{{\mathbb{Z}}}
\newcommand{\N}{{\mathbb{N}}}

\newcommand{\te}[1]{\textnormal{{#1}}}
\newcommand{\set}[2]{{
    \left.\left\{
        {#1}
    \,\right|\,
        {#2}
    \right\}
}}

\def \<{{\langle}}
\def \>{{\rangle}}

\def\({\left(}

\def\){\right)}

\newcommand{\wh}[1]{\widehat{#1}}

\newlength{\dhatheight}

\newcommand{\nord}{\mbox{\scriptsize ${\circ\atop\circ}$}}

\newcommand{\h}{{\frak h}}

\newcommand{\al}{\alpha}

\newcommand{\be}{\beta}

\newcommand{\half}{{\frac{1}{2}}}

\newcommand{\inverse}{^{-1}}

\newcommand{\uz}{_{(0)}}

\newcommand{\E}{{\mathcal{E}}}

\newcommand{\ot}{\otimes}

\newcommand{\vac}{{{\bf 1}}}

\newcommand{\hctvs}[1]{Hausdorff complete linear topological vector space}
\newcommand{\hcta}[1]{Hausdorff complete linear topological algebra}
\newcommand{\ons}[1]{open neighborhood system}

\newcommand{\der}{\mathcal D}

\makeatletter \@addtoreset{equation}{section}

\makeatother \makeatletter

\title{$(G,\chi_{\phi})$-equivariant $\phi$-coordinated quasi modules for nonlocal vertex algebras}

\author{Naihuan Jing$^1$}
\address{Department of Mathematics, North Carolina State University, Raleigh, NC 27695,
USA}
\email{jing@math.ncsu.edu}
\thanks{$^1$Partially supported by NSF of China (No.11531004) and Simons Foundation (No.198129).}

\author{Fei Kong$^2$}
\address{Key Laboratory of Computing and Stochastic Mathematics (Ministry of Education), School of Mathematics and Statistics, Hunan Normal University, Changsha, China 410006} \email{kongmath@hunnu.edu.cn}
\thanks{$^2$Partially supported by NSF of China (No.11701183).}

\author{Haisheng Li$^3$}
\address{Department of Mathematical Sciences, Rutgers University, Camden, NJ 08102, USA}
\email{hli@camden.rutgers.edu}

 \author{Shaobin Tan$^4$}
 \address{School of Mathematical Sciences, Xiamen University,
 Xiamen, China 361005} \email{tans@xmu.edu.cn}
 \thanks{$^4$Partially supported by NSF of China (No.11531004).}

\subjclass[2010]{Primary 17B69, 17B68; Secondary  17B10, 81R10} \keywords{vertex algebras, nonlocal vertex algebras, {$phi$}-coordinated quasi modules, generalized commutator formula, lattice vertex operator algebras}

\begin{document}

\begin{abstract}
In this paper, we study
$(G,\chi_{\phi})$-equivariant $\phi$-coordinated quasi modules for  nonlocal vertex algebras.
Among the main results, we establish several conceptual results, including a generalized commutator formula and
a general construction of weak quantum vertex algebras and their $(G,\chi_{\phi})$-equivariant $\phi$-coordinated quasi modules.
As an application, we also construct (equivariant) $\phi$-coordinated quasi modules for lattice vertex algebras by using Lepowsky's work
on twisted vertex operators.
\end{abstract}

\maketitle

\section{Introduction}
In vertex algebra theory, among the most important notions are  module and twisted module,
which were introduced in the early development (see \cite{FLM}).
Later, a notion of quasi module, generalizing that of module, was introduced in \cite{li-new}
in order to associate vertex algebras to a certain family of infinite-dimensional Lie algebras.
Indeed, with this generalization vertex algebras can be associated to a much wider variety of infinite-dimensional Lie algebras.
To study quasi modules more effectively, vertex algebras with a certain enhanced structure,
called $\Gamma$-vertex algebras,
were studied in \cite{li-gamma} (cf. \cite{li-new}), and meanwhile a notion of
equivariant quasi module for a $\Gamma$-vertex algebra was introduced.
The notion of quasi module actually has a close relation with  that of twisted module;
It was proved in \cite{li-twisted-quasi} that for a general vertex operator algebra $V$
in the sense of \cite{FLM} and \cite{fhl} with a finite order automorphism $\sigma$,
the category of $\sigma$-twisted $V$-modules is canonically isomorphic to a subcategory of equivariant quasi $V$-modules.

In \cite{ek},  Etingof and Kazhdan  introduced a fundamental notion  of quantum vertex operator algebra
in the sense of formal deformation, where a key axiom is called $S$-locality.
With this as one of the main motivations, a  theory of (weak) quantum vertex algebras and
their quasi modules was developed in \cite{li-nonlocal,li-qva2}, where
quantum vertex algebras are natural generalizations of  vertex algebras and vertex super-algebras.
It has been well understood that vertex algebras are analogues and generalizations of
commutative and associative algebras. What were called nonlocal vertex algebras are
noncommutative associative algebra counterparts of vertex algebras, which were studied
independently in \cite{bk}  as field algebras and in \cite{li-g1} as axiomatic $G_1$-vertex algebras.
A weak quantum vertex algebra is simply a nonlocal vertex algebra that satisfies the $S$-locality.

With the goal to associate quantum vertex algebras in the sense of  \cite{li-nonlocal} to quantum affine algebras,
a theory of what were called $\phi$-coordinated quasi modules for quantum vertex algebras was developed
in \cite{li-cmp,li-jmp}.  In this theory, $\phi$ is what was called an associate of the
$1$-dimensional additive formal group, where usual quasi modules
 are simply $\phi$-coordinated quasi modules with $\phi$ taken to be the formal group itself.
By using this new theory with $\phi$ taken to be another particular associate,
weak quantum vertex algebras were associated {\em conceptually} to quantum affine algebras.
It remains to make  this conceptual association explicit.

The current paper is one in a series to establish and study explicit associations of quantum vertex algebras
to various algebras such as quantum affine algebras.
In this paper, we systematically study $\phi$-coordinated quasi modules for general nonlocal vertex algebras
and establish certain general results, laying the foundation for this series of studies.
Among the main results,  we obtain certain generalized commutator relations
for vertex operators on nonlocal vertex algebras and their $\phi$-coordinated quasi modules, and
we give a general construction of weak quantum vertex algebras and $\phi$-coordinated quasi modules.
Furthermore, we construct $\phi$-coordinated quasi modules for vertex algebras associated to nondegenerate even lattices.

Now, we continue the introduction with more technical details.
Let us start with the notion of vertex algebra. For a vertex algebra $V$, among the main ingredients is a linear map
$Y(\cdot,x):\ V\rightarrow (\te{End}V)[[x,x^{-1}]]$, called  the vertex operator map,  where the main axioms are
{\em weak commutativity (namely, locality):} For any $u,v\in V$, there exists a nonnegative integer $k$ such that
\begin{align}
(x_1-x_2)^{k}Y(u,x_1)Y(v,x_2)=(x_1-x_2)^{k}Y(v,x_2)Y(u,x_1)
\end{align}
and {\em weak associativity:} For any $u,v\in V,\ w\in W$, there exists a nonnegative integer $l$ such that
\begin{align}
(x_0+x_2)^{l}Y(u,x_0+x_2)Y(v,x_2)w=(x_0+x_2)^{l}Y(Y(u,x_0)v,x_2)w.
\end{align}
Let $V$ be a vertex algebra. For  a $V$-module $W$ with $Y_{W}(\cdot,x)$
denoting the vertex operator map,  literally the same locality and the same weak associativity hold.
That is, for any $u,v\in V$, there exists a nonnegative integer $k$ such that
\begin{align}
(x_1-x_2)^{k}Y_{W}(u,x_1)Y_{W}(v,x_2)=(x_1-x_2)^{k}Y_{W}(v,x_2)Y_{W}(u,x_1)
\end{align}
and for any $u,v\in V,\ w\in W$, there exists a nonnegative integer $l$ such that
\begin{align}
(x_0+x_2)^{l}Y_{W}(u,x_0+x_2)Y_{W}(v,x_2)w=(x_0+x_2)^{l}Y_{W}(Y(u,x_0)v,x_2)w.
\end{align}
Note that in the definition of a module weak associativity alone is sufficient.

As for a quasi module $(W,Y_{W})$  for a vertex algebra $V$, the key difference is that
 vertex operators $Y_{W}(v,x)$ on $W$ for $v\in V$, instead of being  local, are
 {\em quasi local} in the sense that for $u,v\in V$,
 there exists a nonzero polynomial $p(x_1,x_2)$ such that
 \begin{align}
 p(x_1,x_2)Y_{W}(u,x_1)Y_{W}(v,x_2)=p(x_1,x_2)Y_{W}(v,x_2)Y_{W}(u,x_1).
 \end{align}
 Correspondingly, quasi modules satisfy {\em quasi weak associativity}:
 For any $u,v\in V,\ w\in W$, there exists a nonzero polynomial $q(x_1,x_2)$ such that
 \begin{align}
q(x_0+x_2,x_2)Y_{W}(u,x_0+x_2)Y_{W}(v,x_2)w=q(x_0+x_2,x_2)Y_{W}(Y(u,x_0)v,x_2)w.
\end{align}

Notice that  in the definitions of the notions of module and quasi module (and twisted module as well) for a vertex algebra,
a common phenomenon  is the formal substitution $x_1=x_2+x_0$ in the weak associativity axiom.
Polynomial  $F(x,y)=x+y$ is known as the $1$-dimensional additive formal group (law) and
the role of the formal group played in vertex algebra theory has long been recognized (cf. \cite{bor}).
What is new is a role played by what was called associate of $F(x,y)$.
By definition (see \cite{li-cmp}), an associate is a formal series $\phi(x,z)\in \C((x))[[z]]$, satisfying the condition
$$\phi(x,0)=x,\   \   \  \  \phi(\phi(x,y),z)=\phi(x,y+z).$$
We see that an associate to $F(x,y)$ is the same as a $G$-set to a group $G$.
 It was proved therein that for any $p(x)\in \C((x))$,  the formal series $\phi(x,z)$ defined by
 $\phi(x,z)=e^{zp(x)\frac{d}{dx}}x$
 is an associate of $F(x,y)$ and every associate is of this form.
In particular, taking $p(x)=1$ we get $\phi(x,z)=x+z=F(x,z)$, whereas taking $p(x)=x$, we get $\phi(x,z)=xe^z$.
The essence of \cite{li-cmp} is that  to each associate $\phi(x,z)$, a theory of $\phi$-coordinated
quasi modules for a general nonlocal vertex algebra $V$ is attached.
The main defining  property for $\phi$-coordinated quasi modules is the {\em weak $\phi$-associativity:} For any $u,v\in V$,
there exists a nonzero polynomial $p(x_1,x_2)$ such that
\begin{align}
p(x_1,x_2)Y_{W}(u,x_1)Y_{W}(v,x_2)\in \te{Hom}(W,W((x_1,x_2)))
\end{align}
and
 \begin{align*}
\left(p(x_1,x_2)Y_{W}(u,x_1)Y_{W}(v,x_2)\right)|_{x_1=\phi(x_2,x_0)}
=p(\phi(x_2,x_0),x_2)Y_{W}(Y(u,x_0)v,x_2).
\end{align*}

Let $G$ be a group with  a linear character $\chi$. By {\em a $(G,\chi)$-module
nonlocal vertex algebra} we mean a nonlocal vertex algebra $V$ together with a representation $R$ of $G$
on $V$ such that $R(g){\bf 1}={\bf 1}$ and
\begin{eqnarray*}
R(g)Y(u,x)R(g)^{-1}=Y(R(g)u,\chi(g)x)
\end{eqnarray*}
on $V$  for $g\in G,\ u\in V$. This notion slightly generalizes that of $\Gamma$-vertex algebra,
introduced  in \cite{li-gamma} (cf. \cite{li-new}).
Note that if $\chi$ is the trivial character,  $G$ acts on $V$ as an automorphism group.
Assume that $V$ is a $(G,\chi)$-module nonlocal vertex algebra.
Let $\chi_{\phi}$ be another linear character of $G$ such that
$$\phi(x,\chi(g)z)=\chi_{\phi}(g)\phi (\chi_{\phi}(g)^{-1}x,z)\   \   \mbox{ for }g\in G.$$
We define a $(G,\chi_{\phi})$-equivariant $\phi$-coordinated quasi $V$-module to be
a $\phi$-coordinated quasi $V$-module $(W,Y_{W})$ satisfying the conditions that
$$Y_{W}(R(g)u,x)=Y_{W}(u,\chi_{\phi}(g)^{-1}x)\   \   \   \mbox{ for }g\in G,\ u\in V$$
and that for $u,v\in V$, there exists a nonzero polynomial $f(x)$
whose zeroes are contained in $\chi_{\phi}(G)$ such that
\begin{eqnarray}
f(x_1/x_2)Y_{W}(u,x_1)Y_{W}(v,x_2)\in \te{Hom} (W,W((x_1,x_2))).
\end{eqnarray}
In this paper, we mostly restrict ourselves to the case with $\phi(x,z)=\phi_{r}(x,z):=e^{zx^{r+1}\frac{d}{dx}}x$
where $r\in \Z$.
This very notion with $\phi(x,z)=\phi_{-1}(x,z)=x+z$ is reduced to
the notion of $G$-equivariant quasi module studied in \cite{li-gamma}, and with $\phi(x,z)=\phi_0(x,z)=xe^z$
is reduced to the notion of $G$-equivariant $\phi$-coordinated quasi module studied in \cite{li-jmp}.
As one of the main results of this paper, we obtain
a generalized commutator formula 
for nonlocal vertex algebras generated by vertex operators on vector spaces.

Note that $\phi_{r}$-coordinated modules for vertex algebras have been previously studied in \cite{blp}, where
 a Jacobi type identity and a commutator formula for $\phi_r$-coordinated modules were obtained, and
$\phi_r$-coordinated modules for vertex algebras associated to Novikov algebras by Primc were also studied.

In  vertex algebra theory,  vertex algebras $V_{L}$ associated to nondegenerate
even lattices $L$ (see \cite{bor}, \cite{FLM}) play an important role.
These vertex algebras have been well studied in literature and their
irreducible modules as well as twisted modules have been constructed and classified
(see \cite{Lep}, \cite{FLM}, \cite{dong1}, \cite{dong2}, \cite{dl}, \cite{dl2}; \cite{LL}).
 In this paper, using the work of Lepowsky on twisted vertex operators \cite{Lep},
we  construct $(G,\chi_{\phi})$-equivariant $\phi$-coordinated quasi $V_{L}$-modules
with $G=\<\wh\mu\>$, where $\wh\mu$ is an automorphism  of $V_{L}$,
lifted from a certain isometry $\mu$ of $L$.
Note that  an associative algebra $A(L)$ was introduced and used in \cite{LL} as an essential tool
in the construction of $V_{L}$-modules.
This algebra $A(L)$ also plays a crucial role in the current paper for our construction of $\phi$-coordinated
quasi $V_{L}$-modules.

This paper is organized as follows. In Section 2, we revisit  formal calculus involving associates of
the additive formal group and delta functions.
In Section 3,  we establish several general results on equivariant $\phi$-coordinated quasi modules
for nonlocal vertex algebras. In Section 4, we give a vertex operator construction of
$(\Z_N,\chi_{\phi})$-equivariant $\phi$-coordinated quasi $V_{L}$-modules.

\section{$\phi$-coordinated quasi modules for nonlocal vertex algebras}
In this section, we revisit $\phi$-coordinated quasi modules for nonlocal vertex algebras.
As the main results, we establish several technical results, including a generalized commutator formula
(Theorem \ref{prop:tech-calculation5}) and a general construction of weak quantum vertex algebras and
$\phi$-coordinated quasi modules (Theorem \ref{qva-specialization-2}).

\subsection{Associates of formal groups}
In this subsection,  we first briefly recall from \cite{li-cmp} the notion of associate for the additive formal group (law) and
some basic results, and then we present some technical results on delta functions (Lemma \ref{lem:delta-func-tech-calculation4}).

A {\em one-dimensional formal group (law)} over $\C$ (cf. \cite{Ha-formal-group-book})
is a formal power series $F(x,y)\in \C[[x,y]]$ such that
$$F(0,y)=y,\  \ F(x,0)=x,\   \    F(F(x,y),z)=F(x,F(y,z)).$$
A typical example of one-dimensional formal group is the {\em additive formal group} $F_a(x,y)=x+y$,
simply denoted by $F_{a}$.

Let $F(x,y)$ be a one-dimensional formal group over $\C$.
An {\em associate}  of $F(x,y)$ is a formal series
$\phi(x,z)\in\C((x))[[z]]$, satisfying the condition that
\begin{align*}
  \phi(x,0)=x,\,\,\,\,
  \phi(\phi(x,y),z)=\phi(x,F(y,z)).
\end{align*}

The following classification result was obtained in \cite{li-cmp}:

\begin{prop}\label{prop:classification-assos}
Let $p(x)\in \C((x))$.
Set
\begin{align*}
  \phi(x,z)=e^{zp(x)\frac{d}{dx}}x
  =\sum_{n\ge0}\frac{z^n}{n!}\(p(x)\frac{d}{dx}\)^nx\in\C((x))[[z]].
\end{align*}
Then $\phi(x,z)$ is an associate of $F_a(x,y)$.
On the other hand, every associate of $F_a(x,y)$ is of this form with $p(x)$ uniquely determined.
\end{prop}

\begin{rem}
{\em Notice that taking $p(x)=1$ in Proposition \ref{prop:classification-assos}, we get
$\phi(x,z)=x+z=F_{a}(x,z)$ (the formal group itself), and taking $p(x)=x$ we get
$\phi(x,z)=xe^{z}$.}
\end{rem}

For convenience, for $n\in \N$, set
\begin{eqnarray}
\partial^{n}_{z}=\left(\frac{\partial}{\partial z}\right)^{n}\quad\te{and}
\quad\partial^{(n)}_{z}=\frac{1}{n!}\left(\frac{\partial}{\partial z}\right)^n.
\end{eqnarray}

From now on, we assume that {\em $\phi(x,z)$ is an associate of $F_{a}$ with $\phi(x,z)\ne x$,} that is,
$\phi(x,z)=e^{zp(x)\frac{d}{dx}}x$ with $p(x)\ne 0$.

\begin{rem}\label{rem:about-phi1}
{\em We here recall some basic facts and conventions from \cite[Remark 2.6, Lemma 2.7]{li-cmp}.
 First of all, for every $f(x_1,x_2)\in \C((x_1,x_2))$,  $f(\phi(x,z),x)$ exists in $\C((x))[[z]]$.
 Furthermore, the correspondence
$f(x_1,x)\mapsto f(\phi(x,z),x)$ gives a  ring embedding
$$\pi_{\phi}:\  \C((x_1,x))\rightarrow \C((x))[[z]]\subset \C((x))((z)).$$
Denote by $\C_\ast((x_1,x))$ the fraction field of $\C((x_1,x))$.
Then $\pi_{\phi}$ naturally extends to a field embedding
\begin{align*}
\pi_{\phi}^{*}: \  \C_\ast((x_1,x))\rightarrow \C((x))((z)).
\end{align*}
As a convention,  for any $f(x_1,x)\in\C_\ast((x_1,x))$, we shall view
$f(\phi(x,z),x)$ as an element of $\C((x))((z))$ through this field embedding $\pi_{\phi}^{*}$.}
\end{rem}

\begin{rem}\label{rem:about-phi2}
{\em  Note that for any $f(x_1,x_2)\in\C((x_1,x_2))$,  $f(\phi(x,y),\phi(x,z))$
equals  $$\(\sum_{m,n\ge 0}y^mz^n\frac{(p(x_1)\partial_{x_1})^m}{m!}
\frac{(p(x_2)\partial_{x_2})^n}{n!}f(x_1,x_2)\)|_{x_1=x_2=x},$$
which exists in $\C((x))[[y,z]]$.
Suppose $f(\phi(x,z),\phi(x,y))= 0$. Then $f(\phi(x,z),x)=f(\phi(x,z),\phi(x,0))=0$.
From Remark \ref{rem:about-phi1},
we have $f(x_1,x_2)=0$. Therefore, the correspondence $$f(x_1,x_2)\mapsto f(\phi(x,y),\phi(x,z))$$
is a ring embedding of $\C((x_1,x_2))$ into $\C((x))[[y,z]]$.
Just as with other iota-maps, we have field embeddings
\begin{align}
&\iota_{x,y,z}: \  \C_{*}((x_1,x_2))\rightarrow \C((x))((y))((z)),\\
&\iota_{x,z,y}: \   \C_{*}((x_1,x_2))\rightarrow \C((x))((z))((y)).
\end{align}
As a convention, for $f(x_1,x_2)\in\C_\ast((x_1,x_2))$,
we shall view $f(\phi(x,y),\phi(x,z))$ as an element of the fraction field of
$\C((x))[[y,z]]$.}
\end{rem}

The following are some results about (formal) delta-functions:

\begin{lem}\label{lem:delta-func-tech-calculation4}
Let $c\in \C^\times$.
Set
\begin{align*}
  \Delta_c(x_1,x_2,z)=e^{zp(x_2)\partial_{x_2}}\frac{p(x_1)}{x_1-cx_2}\in\C_\ast((x_1,x_2))[[z]].
\end{align*}
Then
\begin{align}
  &\(\iota_{x_1,x_2,z}-\iota_{x_2,x_1,z}\)\Delta_c(x_1,x_2,z)
  =e^{zp(x_2)\partial_{x_2}} p(x_1)x_1\inverse \delta\(\frac{cx_2}{x_1}\),\label{eq:delta-func1}\\
  &\Delta_c(x_1,\phi(x_2,w),z)=\Delta_c(x_1,x_2,w+z)\label{eq:delta-func1-2},\\
  &\Delta_c(c'x_2,x_2,z)\in\C((x_2))[[z]]\quad\te{for }c'\ne c,\label{eq:delta-func1-3}
\end{align}
and
\begin{align}\label{eq:delta-func2}
  \(\iota_{x,x_1,x_2,z}-\iota_{x,x_2,x_1,z}\) \Delta_c(\phi(x,x_1),\phi(x,x_2),z)
  =
  \left\{
  \begin{array}{ll}
  e^{z\partial_{x_2}} x_1\inverse \delta\(\frac{x_2}{x_1}\)&\te{if }c=1,\\
  0 &\te{if }c\ne 1.
  \end{array}
  \right.
\end{align}
\end{lem}

\begin{proof} Relation \eqref{eq:delta-func1} follows immediately from the fact
\begin{align*}\label{eq:delta-func-temp1}
x_1\inverse\delta\(\frac{cx_2}{x_1}\)=\iota_{x_1,x_2}\(\frac{1}{x_1-cx_2}\)
  -\iota_{x_2,x_1}\(\frac{1}{x_1-cx_2}\).
\end{align*}
Note that
\begin{align*}
 \Delta_c(x_1,x_2,z)=\frac{p(x_1)}{x_1-c\phi(x_2,z)}.
\end{align*}
Using this we get \eqref{eq:delta-func1-2} and \eqref{eq:delta-func1-3}. We also have
\begin{align*}
\Delta_c(\phi(x,x_1),\phi(x,x_2),z)
 =\frac{p(\phi(x,x_1))}{\phi(x,x_1)-c\phi(\phi(x,x_2),z)}
 =\frac{p(\phi(x,x_1))}{\phi(x,x_1)-c\phi(x,x_2+z)}.
\end{align*}
If $c\ne 1$, as $\phi(x,0)-c\phi(x,0)=(1-c)x\ne 0,$
we see that  $\phi(x,x_1)-c\phi(x,x_2)$ is invertible in $\C((x))[[x_1,x_2]]$.
Consequently, $\Delta_c(\phi(x,x_1),\phi(x,x_2),z)\in\C((x))[[x_1,x_2,z]]$ and hence
\begin{align*}
  (\iota_{x,x_1,x_2,z}-\iota_{x,x_2,x_1,z})\Delta_c(\phi(x,x_1),\phi(x,x_2),z)
  =0.
\end{align*}

Now, we consider the case $c=1$.
As $\phi(x,0)=x$, we have $$\phi(x,x_1)-\phi(x,x_2)=(x_1-x_2)F(x,x_1,x_2)$$
 for some $F(x,x_1,x_2)\in \C((x))[[x_1,x_2]]$.
 We see that $F(x,0,0)=p(x)$, so that $F(x,x_1,x_2)$ is an invertible element of $\C((x))[[x_1,x_2]]$.
In view of this, we have
 $$\iota_{x,x_1,x_2}\left(\frac{p(\phi(x,x_1))}{F(x,x_1,x_2)}\right)
 =\iota_{x,x_2,x_1}\left(\frac{p(\phi(x,x_1))}{F(x,x_1,x_2)}\right).$$
 Then we get
\begin{align*}
  &\(\iota_{x,x_1,x_2}-\iota_{x,x_2,x_1}\)
  \frac{p(\phi(x,x_1))}{\phi(x,x_1)-\phi(x,x_2)}\\
  =&\(\iota_{x,x_1,x_2}-\iota_{x,x_2,x_1}\) \frac{1}{x_1-x_2}\cdot
  \frac{p(\phi(x,x_1))}{F(x,x_1,x_2)}\\
 =& \(\iota_{x_1,x_2}\frac{1}{x_1-x_2}-\iota_{x_2,x_1}\frac{1}{x_1-x_2}\)
\iota_{x,x_1,x_2}\left( \frac{p(\phi(x,x_1))}{F(x,x_1,x_2)}\right)\\
=&x_1\inverse\delta\(\frac{x_2}{x_1}\)\iota_{x,x_1,x_2} \frac{p(\phi(x,x_1))}{\frac{\phi(x,x_1)-\phi(x,x_2)}{x_1-x_2}}\\
  =&x_1\inverse\delta\(\frac{x_2}{x_1}\)\iota_{x,x_2}\(\frac{p(\phi(x,x_2))}{\partial_{x_2}\phi(x,x_2)}\)\\
  =&x_1\inverse\delta\(\frac{x_2}{x_1}\),
\end{align*}
noticing that
\begin{align*}
\partial_{x_2} \phi(x,x_2)=e^{x_2p(x)\partial_{x}}(p(x)\partial_{x})x=e^{x_2p(x)\partial_{x}}p(x)=p(\phi(x,x_2)).
\end{align*}
Using this, we get
\begin{align*}
  &\(\iota_{x,x_1,x_2,z}- \iota_{x,x_2,x_1,z}\)\Delta(\phi(x,x_1),\phi(x,x_2),z)\\
=&\(\iota_{x,x_1,x_2,z}- \iota_{x,x_2,x_1,z}\)\frac{p(\phi(x,x_1))}{\phi(x,x_1)-\phi(x,x_2+z)}\\
=&e^{z\partial_{x_2}}\(\iota_{x,x_1,x_2}- \iota_{x,x_2,x_1}\)\frac{p(\phi(x,x_1))}{\phi(x,x_1)-\phi(x,x_2)}\\
=&e^{z\partial_{x_2}} x_1\inverse\delta\(\frac{x_2}{x_1}\),
\end{align*}
as desired.
\end{proof}

Note that
\begin{align*}
&\Delta(x_1,x_2,z)=e^{zp(x_2)\partial_{x_2}}\frac{p(x_1)}{x_1-x_2}
=\sum_{k\ge 0}\frac{1}{k!}z^k\(p(x_2)\partial_{x_2}\)^k\frac{p(x_1)}{x_1-x_2}.
\end{align*}
As an immediate consequence we have:

\begin{coro}\label{k-form}
For $k\in \N$, set
$$F_{k}(x_1,x_2)=\(p(x_2)\partial_{x_2}\)^k\frac{p(x_1)}{x_1-x_2}\in \C_{*}((x_1,x_2)).$$
Then
\begin{align*}
&\(\iota_{x_1,x_2}-\iota_{x_2,x_1}\)F_k(x_1,x_2)
=\(p(x_2)\partial_{x_2}\)^kp(x_1)x_1^{-1}\delta\left(\frac{x_2}{x_1}\right),\nonumber\\
&\(\iota_{x,x_1,x_2}-\iota_{x,x_2,x_1}\)F_{k}(\phi(x,x_1),\phi(x,x_2))
=\partial_{x_2}^kx_1\inverse \delta\(\frac{x_2}{x_1}\).
\end{align*}
\end{coro}

\begin{lem}\label{lem:delta-func-prod-tech-calculation5}
Let $\Gamma=\{c_1,\dots,c_k\}$ be a finite nonempty subset of $\C^\times$.
Set
\begin{align}
  \Delta_{\Gamma}(x_1,x_2,z_1,\dots,z_k)=\prod_{i=1}^k\Delta_{c_i}(x_1,x_2,z_i).
\end{align}
Then
\begin{align}
&\(\iota_{x_1,x_2,z_1,\dots,z_k}-\iota_{x_2,x_1,z_1,\dots,z_k}\)
\Delta_\Gamma(x_1,x_2,z_1,\dots,z_k)\nonumber\\
=&\sum_{i=1}^k\iota_{x_2,z_1,\dots,z_k}\prod_{j\ne i}\Delta_{c_j}(c_i\phi(x_2,z_i),x_2,z_j)
p(x_1)x_1\inverse\delta\(c_i\frac{\phi(x_2,z_i)}{x_1}\).
\end{align}
Furthermore, for $1\le i\le k$ we have
\begin{align}
  \iota_{x_2,z_1,\dots,z_k}\prod_{j\ne i}\Delta_{c_j}(c_i\phi(x_2,z_i),x_2,z_j)
  \in\C((x_2))[[z_1,\dots,z_k]].
\end{align}
\end{lem}

\begin{proof}
Note that by Lemma \ref{lem:delta-func-tech-calculation4} (see \eqref{eq:delta-func1-3}),
for distinct  $c, c'\in\C^\times$, $\Delta_c(c'\phi(x_2,w),x_2,z)$
exists in $\C((x_2))[[z,w]]$, so that
\begin{align}
  &\iota_{x_1,x_2,z,w}\Delta_c(x_1,x_2,z)
    p(x_1)x_1\inverse\delta\(c'\frac{\phi(x_2,w)}{x_1}\)\nonumber\\
  =&\iota_{x_2,z,w}\Delta_c(c'\phi(x_2,w),x_2,z)
    e^{wp(x_2)\partial_{x_2}}p(x_1)x_1\inverse\delta\(c'\frac{x_2}{x_1}\).\label{eq:desc-temp3}
\end{align}
From \eqref{eq:delta-func1}, we have
\begin{align*}
  \(\iota_{x_1,x_2,z_i}-\iota_{x_2,x_1,z_i}\)\Delta_{c_i}(x_1,x_2,z_i)
  =p(x_1)x_1\inverse\delta\(c_i\frac{\phi(x_2,z_i)}{x_1}\)
\end{align*}
for  $1\le i\le k$.  Therefore,
\begin{align}
  &\(\iota_{x_1,x_2,z_1,\dots,z_k}-\iota_{x_2,x_1,z_1,\dots,z_k}\)
\Delta_\Gamma(x_1,x_2,z_1,\dots,z_k)\nonumber\\
=&\sum_{i=1}^k\iota_{x_1,x_2,z_1,\dots,z_k}\(\prod_{j\ne i}
\Delta_{c_j}(x_1,x_2,z_j)\)p(x_1)x_1\inverse\delta\(c_i\frac{\phi(x_2,z_i)}{x_1}\)\label{eq:desc-temp5}\\
=&\sum_{i=1}^k\iota_{x_2,z_1,\dots,z_k}\prod_{j\ne i}\Delta_{c_j}(c_i\phi(x_2,z_i),x_2,z_j)
p(x_1)x_1\inverse\delta\(c_i\frac{\phi(x_2,z_i)}{x_1}\),\nonumber
\end{align}
where the existence of \eqref{eq:desc-temp5} follows from \eqref{eq:desc-temp3}.
The furthermore statement follows immediately from \eqref{eq:delta-func1-3}.
\end{proof}

\subsection{$\phi$-coordinated quasi modules for nonlocal vertex algebras}

Throughout this subsection, we assume that $\phi(x,z)=e^{zp(x)\frac{d}{dx}}x$ is an associate of the formal additive group $F_a$
such that $p(x)\ne 0$, or equivalently $\phi(x,z)\ne x$.
Note that under this assumption,  we have
$$f(\phi(x,z),x)\ne 0\   \   \   \mbox{  for any nonzero }f(x_1,x_2)\in \C[[x_1,x_2]].$$

We start with the notion of nonlocal vertex algebra (see Remark \ref{nlva-definitions}).

\begin{de}\label{def-nlva}
A {\em nonlocal vertex algebra} is a vector space $V$ equipped with a linear map
$$Y(\cdot,x): \  V\rightarrow (\te{End} V)[[x,x^{-1}]];\ v\mapsto Y(v,x)$$
and a distinguished vector ${\bf 1}\in V$, satisfying the conditions that for $u,v\in V$,
\begin{align}
Y(u,x)v\in V((x)),
\end{align}
\begin{align}
Y({\bf 1},x)v=v,\   \   \   \   Y(v,x){\bf 1}\in V[[x]]\   \mbox{ and }\  \lim_{x\rightarrow 0}Y(v,x){\bf 1}=v
\end{align}
and that for $u,v\in V$, there exists a nonnegative integer $k$ such that
\begin{align}\label{nlva-compatibilty}
(x_1-x_2)^{k}Y(u,x_1)Y(v,x_2)\in \te{Hom}(V,V((x_1,x_2)))
\end{align}
and
\begin{align}\label{nlva-associativity}
\left((x_1-x_2)^{k}Y(u,x_1)Y(v,x_2)\right)|_{x_1=x_2+x_0}=
x_0^{k}Y(Y(u,x_0)v,x_2).
\end{align}
\end{de}

\begin{rem}
{\em Let $A(x_1,x_2)\in \te{Hom}(W,W((x_1,x_2)))$ with $W$ a vector space. Then
$$A(x_1,x_2)|_{x_1=x_2+x_0}\  \mbox{ exists in }(\te{Hom}(W,W((x_2)))[[x_0]]$$
and
$$A(x_1,x_2)|_{x_1=x_0+x_2}\  \mbox{ exists in }(\te{Hom}(W,W((x_2)))[[x_0,x_0^{-1}]].$$
We have
\begin{align}
A(x_2+x_0,x_2)=\te{Res}_{x_1}x_1^{-1}\delta\left(\frac{x_2+x_0}{x_1}\right)A(x_1,x_2).
\end{align}
In view of this, under the condition (\ref{nlva-compatibilty}), we have
\begin{align}
&\left((x_1-x_2)^{k}Y(u,x_1)Y(v,x_2)\right)|_{x_1=x_2+x_0}\nonumber\\
=\ &\te{Res}_{x_1}x_1^{-1}\delta\left(\frac{x_2+x_0}{x_1}\right)\left((x_1-x_2)^{k}Y(u,x_1)Y(v,x_2)\right).
\end{align}}
\end{rem}


\begin{lem}\label{two-definitions}
Let $V$ be a nonlocal vertex algebra. Then the following
 {\em weak associativity} holds: For any $u,v,w\in V$, there exists a nonnegative integer $l$ such that
 \begin{align}\label{weak-assoc}
 (x_0+x_2)^{l}Y(u,x_0+x_2)Y(v,x_2)w= (x_0+x_2)^{l}Y(Y(u,x_0)v,x_2)w.
\end{align}
\end{lem}

\begin{proof}
Let $u,v,w\in V$. From definition, there exists $k\in \N$ such that (\ref{nlva-compatibilty}) holds. Then
$$(x_1-x_2)^{k}Y(u,x_1)Y(v,x_2)w\in V((x_1,x_2)).$$
Hence, there exists $l\in \N$ such that
$$x_1^{l}(x_1-x_2)^{k}Y(u,x_1)Y(v,x_2)w\in V[[x_1,x_2]][x_2^{-1}],$$
involving only nonnegative integer powers of $x_1$.
Then
$$\left(x_1^{l}(x_1-x_2)^{k}Y(u,x_1)Y(v,x_2)w\right)|_{x_1=x_2+x_0}
=\left(x_1^{l}(x_1-x_2)^{k}Y(u,x_1)Y(v,x_2)w\right)|_{x_1=x_0+x_2}.$$
Using this and (\ref{nlva-associativity}) we get
\begin{align*}
&(x_2+x_0)^{l}x_0^{k}Y(Y(u,x_0)v,x_2)w\\
=&\left(x_1^{l}(x_1-x_2)^{k}Y(u,x_1)Y(v,x_2)w\right)|_{x_1=x_2+x_0}\\
=&\left(x_1^{l}(x_1-x_2)^{k}Y(u,x_1)Y(v,x_2)w\right)|_{x_1=x_0+x_2}\\
=&(x_0+x_2)^{l}x_0^{k}Y(u,x_0+x_2)Y(v,x_2)w,
\end{align*}
which immediately yields (\ref{weak-assoc}).
\end{proof}

\begin{rem}\label{nlva-definitions}
{\em Note that a notion of nonlocal vertex algebra was defined in  \cite{li-nonlocal} in terms of
weak associativity (\ref{weak-assoc}).
In view of Lemma \ref{two-definitions}, the notion of nonlocal vertex algebra in the sense of
Definition \ref{def-nlva} is theoretically stronger than the other same named notion. }
\end{rem}

Let $V$ be a nonlocal vertex algebra. Define a linear operator $\der$ by $\der (v)= v_{-2}\vac$ for $v\in V$.
Then we have (see \cite{li-nonlocal}):
\begin{align}
[\der, Y(v,x)]=Y(\der (v),x)=\frac{d}{dx}Y(v,x)\   \   \te{for }v\in V.
\end{align}

What were called weak quantum vertex algebras in  \cite{li-nonlocal} form
a special class of nonlocal vertex algebras, where
a {\em weak quantum vertex algebra} is defined by replacing the
compatibility and weak associativity conditions in Definition \ref{def-nlva} with the condition that for any $u,v\in V$,
there exist
$$u^{(i)},v^{(i)}\in V,\ f_i(x)\in \C((x))\  \  (i=1,\dots,r)$$
such that
\begin{align}\label{S-Jacobi}
&x_0^{-1}\delta\left(\frac{x_1-x_2}{x_0}\right)Y(u,x_1)Y(v,x_2)\nonumber\\
&\hspace{1cm}-x_0^{-1}\delta\left(\frac{x_2-x_1}{-x_0}\right)\sum_{i=1}^{r}f_i(-x_2+x_1)Y(v^{(i)},x_2)Y(u^{(i)},x_1)\nonumber\\
=\ &x_1^{-1}\delta\left(\frac{x_2+x_0}{x_1}\right)Y(Y(u,x_0)v,x_2).
\end{align}

\begin{de}
Let $V$ be a nonlocal vertex algebra and let $\phi$ be an associate of $F_{a}$.
A {\em $\phi$-coordinated quasi $V$-module} is a vector space $W$ equipped with a linear map
$$Y_{W}(\cdot,x): \  V\rightarrow (\te{End} W)[[x,x^{-1}]];\ v\mapsto Y_{W}(v,x),$$
satisfying the conditions that
$$Y_{W}(u,x)w\in W((x))\   \   \    \te{for }u\in V,\ w\in W,$$
$$Y_{W}({\bf 1},x)=1_{W} \   \  (\te{the identity operator on }W),$$
and that for any $u,v\in V$, there exists a nonzero $f(x_1,x_2)\in \C[[x_1,x_2]]$ such that
\begin{align}
&f(x_1,x_2)Y_{W}(u,x_1)Y_{W}(v,x_2)\in \te{Hom}(W,W((x_1,x_2))),\\
&\(f(x_1,x_2)Y_{W}(u,x_1)Y_{W}(v,x_2)\)|_{x_1=\phi(x_2,z)}=f(\phi(x_2,z),x_2)Y_{W}(Y(u,z)v,x_2).\label{quasi-weak-assoc}
\end{align}
\end{de}

Let $W$ be a vector space, which is fixed for the rest of this section. Set
$$\E(W)=\te{Hom}(W,W((x))).$$
More generally, for any positive integer $r$, set
\begin{align}
\E^{(r)}(W)=\te{Hom}(W,W((x_1,x_2,\dots,x_r))).
\end{align}
An ordered sequence $\(a_1(x),\dots,a_r(x)\)$ in
$\E(W)$ is said to be {\em quasi-compatible} (see  \cite{li-nonlocal,li-g1}) if
there exists a nonzero series $f(x,y)\in \C[[x,y]]$ such that
\begin{align*}
  \(\prod_{1\le i<j\le r}f(x_i,x_j)\)a_1(x_1)\cdots a_r(x_r)
  \in\E^{(r)}(W).
\end{align*}
A subset $U$ of $\E(W)$ is said to be {\em quasi-compatible} if every finite ordered sequence in $U$ is quasi-compatible.

Suppose that $(\alpha(x),\beta(x))$ is a quasi-compatible pair in $\E(W)$. Define
$$\alpha(x)_{n}^{\phi}\beta(x)\in \E(W) \   \  \te{ for }n\in \Z$$
in terms of generating function
\begin{align}
Y_{\E}^{\phi}(\alpha(x),z)\beta(x)=\sum_{n\in \Z}\alpha(x)_{n}^{\phi}\beta(x) z^{-n-1}
\end{align}
by
\begin{align}
Y_{\E}^{\phi}(\alpha(x),z)\beta(x)=\iota_{x,z}(1/f(\phi(x,z),x))\(f(x_1,x)\alpha(x_1)\beta(x)\)|_{x_1=\phi(x,z)},
\end{align}
where $f(x_1,x_2)$ is any nonzero element of $\C[[x_1,x_2]]$ such that
\begin{align}
f(x_1,x_2)\alpha(x_1)\beta(x_2)\in \te{Hom}(W,W((x_1,x_2)))\  \  (=\E^{(2)}(W)).
\end{align}

A quasi-compatible subspace $U$ of $\E(W)$ is said to be {\em $Y_\E^\phi$-closed} if
\begin{align}
a(x)_{n}^{\phi}b(x)\in U
\end{align}
for all $a(x),b(x)\in U,\ n\in \Z$.

The following result was obtained in \cite{li-cmp}:

\begin{thm}\label{thm:abs-construct-non-h-adic}
Let $W$ be a vector space and  let $V$ be a $Y_\E^\phi$-closed quasi-compatible subspace of $\E(W)$
with $1_{W}\in V$.
Then $(V,Y_\E^\phi,1_{W})$ carries the structure of a nonlocal vertex algebra and
$W$ is a $\phi$-coordinated quasi $V$-module with $Y_W(\al(x),z)=\al(z)$ for $\al(x)\in V$.
On the other hand, for every quasi-compatible subset $U$ of $\E(W)$,
there exists a unique minimal $Y_\E^\phi$-closed quasi-compatible subspace $\<U\>_\phi$
that contains $1_W$ and $U$.
Furthermore, $(\<U\>_\phi,Y_\E^\phi,1_W)$ carries the structure of a nonlocal vertex
algebra and $W$ is a $\phi$-coordinated quasi $\<U\>_\phi$-module.
\end{thm}

\begin{de}
 In view of Theorem \ref{thm:abs-construct-non-h-adic},
 we call a $Y_\E^\phi$-closed quasi-compatible subspace of $\E(W)$ containing $1_{W}$
a {\em nonlocal vertex subalgebra of $\E(W)$.}
\end{de}

The following is an immediate consequence:

\begin{lem}\label{lem:tech-calculation0}
Let  $(\al(x),\be(x))$ be a quasi-compatible pair in $\E(W)$.
Then for any $f(x)\in\C((x))$, $(f(x)\al(x),\be(x))$ is quasi-compatible and
we have
\begin{align*}
  Y_\E^\phi\(f(x)\al(x),z\)\be(x)=f(\phi(x,z))Y_\E^\phi\(\al(x),z\)\be(x).
\end{align*}
\end{lem}

We also have the following result (cf. \cite[Proposition 4.5]{li-cmp}):

\begin{lem}\label{lem:tech-calculation1}
Let  $V$ be a nonlocal vertex subalgebra of $\E(W)$.
Suppose that
\begin{align*}
  \al_i(x),\be_i(x)\in V,\   g_i(x_1,x_2)\in\C_\ast((x_1,x_2))\  \   (i=1,2,\dots,r)
\end{align*}
such that
\begin{align*}
  \sum_{i=1}^r\iota_{x_1,x_2}(g_i(x_1,x_2))\al_i(x_1)\be_i(x_2)\in\te{Hom}(W,W((x_1,x_2))).
\end{align*}
Then
\begin{align}\label{sum-equality}
 & \sum_{i=1}^r\iota_{x,z}(g_i\(\phi(x,z),x\))Y_\E^\phi\(\al_i(x),z\)\be_i(x)\nonumber\\
=&\(\sum_{i=1}^r\iota_{x_1,x}(g_i(x_1,x))\al_i(x_1)\be_i(x)\)\Bigg|_{x_1=\phi(x,z)}.
\end{align}
\end{lem}

\begin{proof} It is similar  to the proof in \cite{li-cmp}.
As $V$ is a quasi-compatible subspace of $\E(W)$, there exists  a nonzero series
$g(x_1,x_2)\in\C[[x_1,x_2]]$ such that
\begin{align*}
  g(x_1,x_2)\al_i(x_1)\be_i(x_2)\in\te{Hom}(W,W((x_1,x_2)))\  \   \,\te{ for }i=1,2,\dots,r.
\end{align*}
From the definition of $Y_\E^\phi$ we have
\begin{align*}
g\(\phi(x,z),x\)Y_\E^\phi\(\al_i(x),z\)\be_i(x)=\(g(x_1,x)\al_i(x_1)\be_i(x)\)|_{x_1=\phi(x,z)}
\end{align*}
for $i=1,2,\dots,r$. Then we get
\begin{align*}
 & g\(\phi(x,z),x\)\sum_{i=1}^r\iota_{x,z}(g_i\(\phi(x,z),x\))Y_\E^\phi\(\al_i(x),z\)\be_i(x)\\
 =&\sum_{i=1}^r\iota_{x,z}(g_i\(\phi(x,z),x\))\(g(x_1,x)\al_i(x_1)\be_i(x)\)|_{x_1=\phi(x,z)}\\
 =&\(g(x_1,x)\sum_{i=1}^r\iota_{x_1,x}(g_i(x_1,x))\al_i(x_1)\be_i(x)\)\Bigg|_{x_1=\phi(x,z)}\\
 =&g\(\phi(x,z),x\)\(\sum_{i=1}^r\iota_{x_1,x}(g_i(x_1,x))\al_i(x_1)\be_i(x)\)\Bigg|_{x_1=\phi(x,z)}.
\end{align*}
Noticing that both $\sum_{i=1}^r\iota_{x,z}(g_i\(\phi(x,z),x\))Y_\E^\phi\(\al_i(x),z\)\be_i(x)$
and
\begin{align*}
  \(\sum_{i=1}^r\iota_{x_1,x}(g_i(x_1,x))\al_i(x_1)\be_i(x)\)\Bigg|_{x_1=\phi(x,z)}\  \te{ lie in }\(\E(W)\)((z)),
\end{align*}
which  is a vector space over $\C((x))((z))$, and that $g(\phi(x,z),x)$ is a nonzero element of  $\C((x))((z))$,
we immediately get (\ref{sum-equality}).
\end{proof}

On the other hand,  we have:

\begin{lem}\label{lem:tech-calculation2}
Let $V$ be a nonlocal vertex subalgebra of $\E(W)$.
Suppose that
\begin{align*}
  \al_i(x),\be_i(x)\in V, \   g_i(x_1,x_2)\in\C_\ast((x_1,x_2))\  \   (i=1,2,\dots,r)
\end{align*}
such that
\begin{align*}
  \sum_{i=1}^r\iota_{x_1,x_2}(g_i(x_1,x_2))\al_i(x_1)\be_i(x_2)\in\te{Hom}(W,W((x_1,x_2))).
\end{align*}
Then
\begin{eqnarray*}
 &&\sum_{i=1}^r\iota_{x,x_1,x_2}(g_i\(\phi(x,x_1),\phi(x,x_2)\))
  Y_\E^\phi\(\al_i(x),x_1\)Y_\E^\phi\(\be_i(x),x_2\)\nonumber\\
  &&\   \   \in \te{Hom}\(V,V((x))((x_1,x_2))\).
\end{eqnarray*}
\end{lem}

\begin{proof} As $V$ is a quasi-compatible subspace of $\E(W)$, from the definition of $Y_\E^\phi$,
 there exists a positive integer $k$ such that
\begin{align*}
 x_0^k Y_\E^\phi\(\al_i(x),x_0\)\be_i(x)\in V[[x_0]]
\end{align*}
for all $1\le i\le r$.
Let $\theta(x)\in V$ be arbitrarily fixed. By the weak associativity of the nonlocal vertex algebra $V$,
we can update $k$ so that we also have
\begin{align}\label{eq:weak-asso}
  &(x_0+x_2)^k Y_\E^\phi\(\al_i(x),x_0+x_2\)Y_\E^\phi\(\be_i(x),x_2\)\theta(x)\nonumber\\
  =&(x_0+x_2)^k
  Y_\E^\phi\(Y_\E^\phi\(\al_i(x),x_0\)\be_i(x),x_2\)\theta(x),
\end{align}
which implies that the common quantity on both sides lies in $V((x_0,x_2))$.  Then
\begin{align*}
  &x_1^k(x_1-x_2)^kY_\E^\phi\(\al_i(x),x_1\)Y_\E^\phi\(\be_i(x),x_2\)\theta(x)\\
=&\((x_0+x_2)^kx_0^k
  Y_\E^\phi\(Y_\E^\phi\(\al_i(x),x_0\)\be_i(x),x_2\)\theta(x)\)|_{x_0=x_1-x_2}
\end{align*}
for all $1\le i\le r$.
Hence, we have
\begin{align}\label{eq:weak-asso-temp1}
&  x_1^k(x_1-x_2)^k \sum_{i=1}^r \iota_{x,x_1,x_2}(g_i(\phi(x,x_1),\phi(x,x_2)))
  Y_\E^\phi\(\al_i(x),x_1\)Y_\E^\phi\(\be_i(x),x_2\)\theta(x)\nonumber\\
  =&\Bigg(x_0^k (x_0+x_2)^k\sum_{i=1}^r \iota_{x,x_2,x_0}(g_i(\phi(x,x_0+x_2),\phi(x,x_2)))\nonumber\\
&\qquad\qquad\times
  Y_\E^\phi\(Y_\E^\phi\(\al_i(x),x_0\)\be_i(x),x_2\)\theta(x)\Bigg)\Bigg|_{x_0=x_1-x_2}.
\end{align}

Recall again that  $g_i(\phi(x,x_0),x)$ exists in $\C((x))((x_0))$.
Then $g_i(\phi(\phi(x,x_2),x_0),\phi(x,x_2))$
exists in $\C((x))[[x_2]]((x_0))\subset \C((x))((x_0,x_2))$.
From Lemma \ref{lem:tech-calculation0} and \eqref{eq:weak-asso}, we have
\begin{align*}
&(x_0+x_2)^k  Y_\E^\phi\(g_i(\phi(x,x_0),x)Y_\E^\phi\(\al_i(x),x_0\)\be_i(x),x_2\)
    \theta(x)\\
=&(x_0+x_2)^k g_i(\phi(\phi(x,x_2),x_0),\phi(x,x_2))
    Y_\E^\phi\(Y_\E^\phi\(\al_i(x),x_0\)\be_i(x),x_2\)\theta(x),
\end{align*}
which implies that the common quantity on both sides lies in $V((x))((x_0,x_2))$.
Then
\begin{align*}
&(x_0+x_2)^k  Y_\E^\phi\(g_i(\phi(x,x_0),x)Y_\E^\phi\(\al_i(x),x_0\)\be_i(x),x_2\)
    \theta(x)\\
=&(x_0+x_2)^k g_i(\phi(x,x_0+x_2),\phi(x,x_2))
    Y_\E^\phi\(Y_\E^\phi\(\al_i(x),x_0\)\be_i(x),x_2\).
\end{align*}
By Lemma \ref{lem:tech-calculation1}, we have
\begin{align*}
&\sum_{i=1}^rg_i(\phi(x,x_0),x)Y_\E^\phi\(\al_i(x),x_0\)\be_i(x)\\
=&\(\sum_{i=1}^rg_i(x_1,x)\al_i(x_1)\be_i(x)\)\Bigg|_{x_1=\phi(x,x_0)}
\in V((x))[[x_0]].
\end{align*}
Then
\begin{align*}
    (x_0+x_2)^k
    Y_\E^\phi\(\sum_{i=1}^rg_i(\phi(x,x_0),x)
    Y_\E^\phi\(\al_i(x),x_0\)\be_i(x),x_2\)\theta(x)\in V((x))[[x_0]]((x_2)).
\end{align*}
Consequently,
\begin{align*}
\((x_0+x_2)^k
    Y_\E^\phi\(\sum_{i=1}^rg_i(\phi(x,x_0),x)
    Y_\E^\phi\(\al_i(x),x_0\)\be_i(x),x_2\)\theta(x)\)\Bigg|_{x_0=x_1-x_2}
\end{align*}
lies in $V((x))((x_1,x_2))$.
Combining this with \eqref{eq:weak-asso-temp1}, we get
\begin{align*}
&  x_1^k(x_1-x_2)^k \sum_{i=1}^r g_i(\phi(x,x_1),\phi(x,x_2))
  Y_\E^\phi\(\al_i(x),x_1\)Y_\E^\phi\(\be_i(x),x_2\)\theta(x)\nonumber\\
  =&(x_1-x_2)^k \Bigg((x_0+x_2)^k\sum_{i=1}^r g_i(\phi(x,x_0+x_2),\phi(x,x_2))\nonumber\\
&\qquad\qquad\qquad\times
  Y_\E^\phi\(Y_\E^\phi\(\al_i(x),x_0\)\be_i(x),x_2\)\theta(x)\Bigg)\Bigg|_{x_0=x_1-x_2}.
\end{align*}
Hence,
\begin{align*}
  &x_1^k\sum_{i=1}^r g_i(\phi(x,x_1),\phi(x,x_2))
  Y_\E^\phi\(\al_i(x),x_1\)Y_\E^\phi\(\be_i(x),x_2\)\theta(x)\nonumber\\
  =&\Bigg(
    (x_0+x_2)^k
    Y_\E^\phi\(\sum_{i=1}^rg_i(\phi(x,x_0),x)Y_\E^\phi\(\al_i(x),x_0\)\be_i(x),x_2\)\theta(x)
  \Bigg)\Bigg|_{x_0=x_1-x_2}
\end{align*}
lies in $V((x))((x_1,x_2))$.
Therefore, we have
\begin{align*}
\sum_{i=1}^r g_i(\phi(x,x_1),\phi(x,x_2))
  Y_\E^\phi\(\al_i(x),x_1\)Y_\E^\phi\(\be_i(x),x_2\)\theta(x)
  \in V((x))((x_1,x_2)),
\end{align*}
as desired.
\end{proof}

We have the following technical result (cf. \cite[Proposition 5.3]{li-cmp}):

\begin{prop}\label{prop:tech-calculation3}
Let $W$ be a vector space  and let $V$ be a nonlocal vertex subalgebra of $\E(W)$.
Suppose
\begin{align*}
  \al_i(x),\be_i(x),\mu_j(x),\nu_j(x)\in V, \
  g_i(x_1,x_2),h_j(x_1,x_2)\in\C_\ast((x_1,x_2))
\end{align*}
for $1\le i\le r,\  1\le j\le s$ such that the following relation holds on $W$:
\begin{align}\label{eq:S-local-temp0001}
  \sum_{i=1}^r\iota_{x_1,x_2}(g_i(x_1,x_2))\al_i(x_1)\be_i(x_2)=\sum_{j=1}^s\iota_{x_2,x_1}(h_j(x_2,x_1))\mu_j(x_2)\nu_j(x_1).
\end{align}
Then the following relation holds on $V$:
\begin{align}\label{adjoint-gcomm}
 & \sum_{i=1}^r \iota_{x,x_1,x_2}(g_i(\phi(x,x_1),\phi(x,x_2)))
    Y_\E^\phi\(\al_i(x),x_1\)Y_\E^\phi\(\be_i(x),x_2\)\nonumber\\
  =&\sum_{j=1}^s \iota_{x,x_2,x_1}(h_i(\phi(x,x_2),\phi(x,x_1)))
    Y_\E^\phi\(\mu_j(x),x_2\)Y_\E^\phi\(\nu_j(x),x_1\).
\end{align}
\end{prop}

\begin{proof} Recall from Remark \ref{rem:about-phi2}  that  $\iota_{x,x_1,x_2}
    g_i(\phi(x,x_1),\phi(x,x_2))\in \C((x))((x_1))((x_2))$
and $\iota_{x,x_2,x_1}
    h_j(\phi(x,x_2),\phi(x,x_1))\in \C((x))((x_2))((x_1))$.
Let $\theta(x)\in V$. As $V$ is quasi-compatible,
there exists nonzero  $f(x_1,x_2)\in\C[[x_1,x_2]]$ such that
\begin{align*}
  &f(x_1,x_2)\be_i(x_1)\theta(x_2),\   \    f(x_1,x_2)\nu_j(x_1)\theta(x_2)\in\E^{(2)}(W),\\
  &f(x_1,x_2)f(x_1,x_3)f(x_2,x_3)\al_i(x_1)\be_i(x_2)\theta(x_3)\in\E^{(3)}(W),\\
  &f(x_1,x_2)f(x_1,x_3)f(x_2,x_3)\mu_j(x_1)\nu_j(x_2)\theta(x_3)
  \in\E^{(3)}(W)
\end{align*}
for $1\le i\le r,$ $1\le j\le s$.
From \cite[Lemma 4.7]{li-cmp}, we have
\begin{align*}
  f(&\phi(x,x_1),\phi(x,x_2))f(\phi(x,x_1),x)f(\phi(x,x_2),x)
  Y_\E^\phi\(\al_i(x),x_1\)Y_\E^\phi\(\be_i(x),x_2\)\theta(x)\\
&  =\(\prod_{1\le a<b\le 3}f(y_a,y_b)\al_i(y_1)\be_i(y_2)\theta(y_3)\)
    \Bigg|_{y_1=\phi(x,x_1),y_2=\phi(x,x_2),y_3=x},\\
  f(&\phi(x,x_1),\phi(x,x_2))f(\phi(x,x_1),x)f(\phi(x,x_2),x)
  Y_\E^\phi\(\mu_j(x),x_1\)Y_\E^\phi\(\nu_j(x),x_2\)\theta(x)\\
&  =\(\prod_{1\le a<b\le 3}f(y_a,y_b)\mu_j(y_1)\nu_j(y_2)\theta(y_3)\)
    \Bigg|_{y_1=\phi(x,x_1),y_2=\phi(x,x_2),y_3=x}
\end{align*}
for $1\le i\le r$ and $1\le j\le s$.
Then we have
\begin{align}\label{eq:temp10001}
 & f(\phi(x,x_1),\phi(x,x_2))f(\phi(x,x_1),x)f(\phi(x,x_2),x)
  \sum_{i=1}^r g_i(\phi(x,x_1),\phi(x,x_2))\nonumber\\
&\quad\times
  Y_\E^\phi\(\al_i(x),x_1\)Y_\E^\phi\(\be_i(x),x_2\)\theta(x)\nonumber\\
=&\(\prod_{1\le a<b\le 3}f(y_a,y_b)\sum_{i=1}^r
    g_i(y_1,y_2)\al_i(y_1)\be_i(y_2)\theta(y_3)\)
    |_{y_1=\phi(x,x_1),y_2=\phi(x,x_2),y_3=x}\nonumber\\
=&\(\prod_{1\le a<b\le 3}f(y_a,y_b)\sum_{j=1}^s
    h_j(y_2,y_1)\mu_j(y_2)\nu_j(y_2)\theta(y_3)\)
    |_{y_1=\phi(x,x_1),y_2=\phi(x,x_2),y_3=x}\nonumber\\
= &f(\phi(x,x_1),\phi(x,x_2))f(\phi(x,x_1),x)f(\phi(x,x_2),x)
\sum_{j=1}^sh_i(\phi(x,x_2),\phi(x,x_1))\nonumber\\
&\quad\times
    Y_\E^\phi\(\mu_j(x),x_2\)Y_\E^\phi\(\nu_j(x),x_1\)\theta(x).
\end{align}
Note that the relation \eqref{eq:S-local-temp0001}
implies that the common quantity on the both sides lies in $V((x_1,x_2))$.
Similarly, with Lemma \ref{lem:tech-calculation2} we conclude that both
\begin{align*}
\sum_{i=1}^r \iota_{x,x_1,x_2}g_i(\phi(x,x_1),\phi(x,x_2))
Y_\E^\phi\(\al_i(x),x_1\)Y_\E^\phi\(\be_i(x),x_2\)\theta(x)
\end{align*}
and
\begin{align*}
\sum_{j=1}^s\iota_{x,x_2,x_1}h_i(\phi(x,x_2),\phi(x,x_1))
Y_\E^\phi\(\mu_j(x),x_2\)Y_\E^\phi\(\nu_j(x),x_1\)\theta(x)
\end{align*}
lie in $V((x))((x_1,x_2))\subset V((x))((x_1))((x_2))$ which is a vector space over $\C((x))((x_1))((x_2))$.
Notice that
\begin{align*}
0\ne f(\phi(x,x_1),\phi(x,x_2))f(\phi(x,x_1),x)f(\phi(x,x_2),x)\in\C((x))((x_1))((x_2)).
\end{align*}
Then, from \eqref{eq:temp10001} we get
\begin{align*}
&\sum_{i=1}^r \iota_{x,x_1,x_2}g_i(\phi(x,x_1),\phi(x,x_2))
Y_\E^\phi\(\al_i(x),x_1\)Y_\E^\phi\(\be_i(x),x_2\)\theta(x)\\
=&\sum_{j=1}^s\iota_{x,x_2,x_1}h_i(\phi(x,x_2),\phi(x,x_1))
Y_\E^\phi\(\mu_j(x),x_2\)Y_\E^\phi\(\nu_j(x),x_1\)\theta(x).
\end{align*}
Thus, (\ref{adjoint-gcomm}) holds.
\end{proof}

As the main result of this section we have:

\begin{thm}\label{prop:tech-calculation5}
Let $W$ be a vector space and let $V$ be a nonlocal vertex subalgebra of $\E(W)$.
Suppose
\begin{align*}
 \al_i(x),\be_i(x), \mu_j(x),\nu_j(x)\in V,
 \  g_i(x_1,x_2),h_j(x_1,x_2)\in\C_\ast((x_1,x_2))
\end{align*}
for $1\le i\le r,\ 1\le j\le s$ such that the following relation holds on $W$:
\begin{align}\label{alphabeta=gamma}
  \sum_{i=1}^r&\iota_{x_1,x_2}\(g_i(x_1,x_2)\)\al_i(x_1)\be_i(x_2)
  -\sum_{j=1}^s\iota_{x_2,x_1}\(h_j(x_1,x_2)\)
    \mu_j(x_2)\nu_j(x_1)\nonumber\\
&=\sum_{i=0}^t\sum_{k=1}^N
  \gamma_{k,i}(x_2)\frac{1}{i!}\(p(x_2)\partial_{x_2}\)^ip(x_1)x_1\inverse \delta\(c_k\frac{x_2}{x_1}\),
\end{align}
where $N,t\ge 0$ and $c_0,\dots,c_N$ are distinct nonzero complex numbers with $c_0=1$
and $\gamma_{k,i}(x)\in V$ for $0\le i\le t$, $0\le k\le N$.
Then
\begin{align}\label{YEphialpha-beta-gamma}
& \sum_{i=1}^r\iota_{x,x_1,x_2}\(g_i(\phi(x,x_1),\phi(x,x_2))\)
  Y_\E^\phi(\al_i(x),x_1)Y_\E^\phi(\be_i(x),x_2)\nonumber\\
 & \   \   -\sum_{j=1}^s\iota_{x,x_2,x_1}
    \(h_j(\phi(x,x_1),\phi(x,x_2))\)
    Y_\E^\phi(\mu_j(x),x_2)Y_\E^\phi(\nu_j(x),x_1)\nonumber\\
  =&\sum_{i=0}^tY_\E^\phi(\gamma_{0,i}(x),x_2)\frac{1}{i!}\partial_{x_2}^ix_1\inverse \delta\(\frac{x_2}{x_1}\).
\end{align}
\end{thm}

\begin{proof}  Notice that we have
\begin{align*}
  &(x_1-cx_2)^{i+1}\(p(x_2)\partial_{x_2}\)^ip(x_1)x_1\inverse \delta\(\frac{cx_2}{x_1}\)=0
\end{align*}
for any $c\in \C^{\times}$ and $i\in \N$ and that
\begin{align*}
\(p(x_2)\partial_{x_2}\)^ip(x_1)x_1\inverse \delta\(\frac{x_2}{x_1}\)
=\iota_{x_1,x_2}F_{i}(x_1,x_2)-\iota_{x_2,x_1}F_{i}(x_1,x_2),
\end{align*}
where $F_{i}(x_1,x_2)=\(p(x_2)\partial_{x_2}\)^i(p(x_1)/(x_1-x_2))$ for $i\ge 0$.
Set $$f(x_1,x_2)=\prod_{k=1}^N(x_1-c_kx_2)^{t+1}.$$
Then we have
\begin{align*}
 & f(x_1,x_2)\sum_{i=0}^t\sum_{k=0}^N
  \gamma_{k,i}(x_2)\frac{1}{i!}\(p(x_2)\partial_{x_2}\)^ip(x_1)x_1\inverse \delta\(c_k\frac{x_2}{x_1}\)\nonumber\\
 =&f(x_1,x_2)\sum_{i=0}^t
  \gamma_{0,i}(x_2)\frac{1}{i!}\(p(x_2)\partial_{x_2}\)^ip(x_1)x_1\inverse \delta\(\frac{x_2}{x_1}\)\nonumber\\
  =& f(x_1,x_2)\sum_{i=0}^t
 \frac{1}{i!} \left(\iota_{x_1,x_2}(F_{i}(x_1,x_2)) -\iota_{x_2,x_1}(F_{i}(x_1,x_2))\right) \gamma_{0,i}(x_2).
 \end{align*}
Using this and (\ref{alphabeta=gamma}) we get
\begin{align*}
 & f(x_1,x_2)\sum_{i=1}^r\iota_{x_1,x_2}\(g_i(x_1,x_2)\)\al_i(x_1)\be_i(x_2)-
 f(x_1,x_2)\sum_{i=0}^t
 \frac{1}{i!} \iota_{x_1,x_2}F_{i}(x_1,x_2)  \gamma_{0,i}(x_2) \nonumber\\
  =&f(x_1,x_2)\sum_{j=1}^s\iota_{x_2,x_1}\(h_j(x_1,x_2)\)
    \mu_j(x_2)\nu_j(x_1)-
    f(x_1,x_2)\sum_{i=0}^t
 \frac{1}{i!}\iota_{x_2,x_1}(F_{i}(x_1,x_2)) \gamma_{0,i}(x_2).
\end{align*}
Then by Proposition \ref{prop:tech-calculation3}, we have
\begin{align*}
 & \sum_{i=1}^r\iota_{x,x_1,x_2}\(f(\phi(x,x_1),\phi(x,x_2))g_i(\phi(x,x_1),\phi(x,x_2))\)
 Y_{\E}^{\phi}(\al_i(x),x_1)Y_{\E}^{\phi}(\be_i(x),x_2)
 \\
 &\   \   \       -\sum_{i=0}^t \iota_{x,x_1,x_2}f(\phi(x,x_1),\phi(x,x_2))
 \frac{1}{i!} \iota_{x,x_1,x_2}F_{i}(\phi(x,x_1),\phi(x,x_2)) Y_{\E}^{\phi}( \gamma_{0,i}(x),x_2) \nonumber\\
  =&\sum_{j=1}^s\iota_{x,x_2,x_1}\(f(\phi(x,x_1),\phi(x,x_2))h_j(\phi(x,x_1),\phi(x,x_2))\)
    Y_{\E}^{\phi}(\mu_j(x),x_2)Y_{\E}^{\phi}(\nu_j(x),x_1)\\
    &\   \    \    -\sum_{i=0}^t \iota_{x,x_2,x_1}f(\phi(x,x_1),\phi(x,x_2))
 \frac{1}{i!}\iota_{x,x_2,x_1}F_{i}(\phi(x,x_1),\phi(x,x_2))
    Y_{\E}^{\phi}(\gamma_{0,i}(x),x_2).
\end{align*}
Note that with $f(x_1,x_2)\in \C[x_1,x_2]$,
 we have  $f(\phi(x,x_1),\phi(x,x_2))\in \C((x))[[x_1,x_2]]$.
As
\begin{align*}
  f(\phi(x,0),\phi(x,0))=f(x,x)=x^{N(t+1)}\prod_{k=1}^N(1-c_k)^{t+1}\ne 0,
\end{align*}
$f(\phi(x,x_1),\phi(x,x_2))$ is invertible in $\C((x))[[x_1,x_2]]$.
Then by cancellation we get
\begin{align*}
 & \sum_{i=1}^r\iota_{x,x_1,x_2}\(g_i(\phi(x,x_1),\phi(x,x_2))\)
 Y_{\E}^{\phi}(\al_i(x),x_1)Y_{\E}^{\phi}(\be_i(x),x_2)\\
 &\   \   \       -\sum_{i=0}^t
 \frac{1}{i!} \iota_{x,x_1,x_2}(F_{i}(\phi(x,x_1),\phi(x,x_2))) Y_{\E}^{\phi}( \gamma_{0,i}(x),x_2) \nonumber\\
  =&\sum_{j=1}^s\iota_{x,x_2,x_1}\(h_j(\phi(x,x_1),\phi(x,x_2))\)
    Y_{\E}^{\phi}(\mu_j(x),x_2)Y_{\E}^{\phi}(\nu_j(x),x_1)\\
    &\    \    -\sum_{i=0}^t
 \frac{1}{i!}\iota_{x,x_2,x_1}(F_{i}(\phi(x,x_1),\phi(x,x_2)))
    Y_{\E}^{\phi}(\gamma_{0,i}(x),x_2).
\end{align*}
Combining this relation with Corollary \ref{k-form},
 we obtain (\ref{YEphialpha-beta-gamma}).
\end{proof}

Next, we consider a special form of Theorem \ref{prop:tech-calculation5}.
For the rest of this section, we  assume
\begin{eqnarray}\label{Psi-psi-element}
\Psi(z_1,z_2)\in \C[z_1^{\pm 1},z_2^{\pm 1}],\    \    \    \   \psi(x)\in \C[[x]] \  \mbox{ with }\psi'(0)=1
\end{eqnarray}
such that
\begin{eqnarray}\label{Psi-phi-psi-1}
 \Psi(\phi(x,x_1),\phi(x,x_2))=\psi(x_1-x_2).
\end{eqnarray}
As $\phi(x,0)=x$, we have
\begin{eqnarray}\label{Psi-phi-psi}
 \Psi(\phi(x,z),x)=\psi(z).
\end{eqnarray}
This implies $\phi(x,z)\ne x$ as $\psi'(z)\ne 0$ from the assumption above.

\begin{rem}\label{phi_r}
{\em  For $r\in \Z$, set
\begin{align}
\phi_r(x,z)=e^{zx^{r+1}\frac{d}{dx}}x.
\end{align}
From \cite{fhl}, we have
\begin{eqnarray*}
\phi_r(x,z)=\begin{cases}x\(1-rzx^{r}\)^{\frac{-1}{r}}&\  \  \mbox{ if }r\ne 0\\
xe^z& \  \  \mbox{ if }r= 0.
\end{cases}
\end{eqnarray*}
Set
\begin{eqnarray}
\Psi_r(z_1,z_2)=\begin{cases}-\frac{1}{r}(z_1^{-r}-z_2^{-r})&\  \  \mbox{ if }r\ne 0\\
z_1/z_2&\  \  \mbox{ if }r= 0
\end{cases}
\end{eqnarray}
and
\begin{eqnarray}
\psi_r(z)=\begin{cases}z&\  \  \mbox{ if }r\ne 0\\
e^{z}& \  \  \mbox{ if }r= 0.
\end{cases}
\end{eqnarray}
Then we have $\Psi_r(z_1,z_2)\in \C[z_1^{\pm 1},z_2^{\pm 1}]$ and $\psi_r(z)\in \C[[z]]$ with $\psi_r'(0)=1$ such that
\begin{eqnarray*}
 \Psi_r\(\phi_r(x,x_1),\phi_r(x,x_2)\)=\psi_r(x_1-x_2).
 \end{eqnarray*}}
\end{rem}


Recall that $\C(x)$ is the fraction field of $\C[x]$.
Let $\C_{*}(x)$ denote the fraction field of $\C[[x]]$, which is isomorphic to $\C((x))$.
Denote the canonical isomorphism by $\iota_{x,0}$:
\begin{eqnarray}
\iota_{x,0}: \   \    \C_{*}(x)\rightarrow \C((x)).
\end{eqnarray}
As $\psi(x)\in \C[[x]]$ with $\psi'(x)\ne 0$, we have $q(\psi(x))\ne 0$ for any nonzero $q(z)\in \C[z]$
(since $\C$ is algebraically closed and $\psi(x)\notin \C$).
Consequently, for any $Q(z)\in \C(z)$, $Q(\psi(x))$ is a well defined element of $\C_{*}(x)$.
Then we have a field embedding
\begin{eqnarray}
 \iota_{z=\psi(x)}: \C(z)\rightarrow \C((x)),
 \end{eqnarray}
 which is defined by
$$\iota_{z=\psi(x)}(Q(z))=\iota_{x,0}(Q(\psi(x)))\   \   \   \mbox{ for } Q(z)\in \C(z).$$
Note that
$$Q(\Psi(x_1,x_2))\in \C(x_1,x_2)\   \   \mbox{ for }Q(z)\in \C(z).$$
 Under this setting with $\Psi$ and $\psi$ satisfying the conditions (\ref{Psi-psi-element}) and (\ref{Psi-phi-psi-1}),
 we immediately have:

\begin{thm}\label{qva-specialization}
Let $W$ be a vector space and let $V$ be a nonlocal vertex subalgebra of $\E(W)$. Suppose
\begin{align*}
 \al_i(x),\be_i(x), \mu_j(x),\nu_j(x)\in V,
 \  f_i(z),g_j(z)\in\C(z)
\end{align*}
for $1\le i\le r,\ 1\le j\le s$ such that the following relation holds on $W$:
\begin{align}\label{alphabeta=gamma}
  \sum_{i=1}^r&\iota_{x_1,x_2}\(f_i(\Psi(x_1,x_2))\)\al_i(x_1)\be_i(x_2)
  -\sum_{j=1}^s\iota_{x_2,x_1}\(g_j(\Psi(x_1,x_2))\)
    \mu_j(x_2)\nu_j(x_1)\nonumber\\
&=\sum_{i=0}^t\sum_{k=1}^N
  \gamma_{k,i}(x_2)\frac{1}{i!}\(p(x_2)\partial_{x_2}\)^ip(x_1)x_1\inverse \delta\(c_k\frac{x_2}{x_1}\),
\end{align}
where $N,t\ge 0$ and $c_0,\dots,c_N$ are distinct nonzero complex numbers with $c_0=1$
and $\gamma_{k,i}(x)\in V$ for $0\le i\le t$, $0\le k\le N$.
Then
\begin{align}\label{YEphialpha-gcommutator-qva}
& \sum_{i=1}^r\iota_{x_1,x_2}\(f_i(\psi(x_1-x_2))\)
  Y_\E^\phi(\al_i(x),x_1)Y_\E^\phi(\be_i(x),x_2)\nonumber\\
 & \   \   -\sum_{j=1}^s\iota_{x_2,x_1}
    \(g_j(\psi(x_1-x_2))\)
    Y_\E^\phi(\mu_j(x),x_2)Y_\E^\phi(\nu_j(x),x_1)\nonumber\\
  =\  &\sum_{i=0}^tY_\E^\phi(\gamma_{0,i}(x),x_2)\frac{1}{i!}\partial_{x_2}^ix_1\inverse \delta\(\frac{x_2}{x_1}\).
\end{align}
\end{thm}

Let $\phi, \Psi,\psi$ be given as before. Motivated by Theorem \ref{qva-specialization}
we formulate the following notion:

\begin{de}
Let $W$ be a vector space.
A subset $U$ of $\E(W)$ is said to be {\em quasi $(\phi,S)$-local} if for any $a(x),b(x)\in U$, there exist
$$u_i(x),v_i(x)\in U,\  f_i(z)\in \C(z)\  \  (i=1,\dots,r)$$
and a nonzero polynomial $q(x_1,x_2)\in \C[x_1,x_2]$ such that
\begin{eqnarray}\label{phi-S-locality}
q(x_1,x_2)a(x_1)b(x_2)=q(x_1,x_2)\sum_{i=1}^r\iota_{x_2,x_1}f_i(\Psi(x_1,x_2))u_i(x_2)v_i(x_1).
\end{eqnarray}
We define a notion of $(\phi,S)$-local subset (without the prefix ``quasi'') by simply changing the phrase
``a nonzero polynomial $q(x_1,x_2)\in \C[x_1,x_2]$'' to ``a Laurent polynomial $q(x_1,x_2)=\Psi(x_1,x_2)^k$ with $k\in \N$''.
\end{de}

\begin{thm}\label{qva-specialization-2}
Let $W$ be a vector space and let $U$ be a quasi $(\phi,S)$-local subset of $\E(W)$. Then
$U$ is quasi compatible and the nonlocal vertex algebra $\<U\>_{\phi}$ generated  in $\E(W)$ by $U$
is a weak quantum vertex algebra.
\end{thm}

\begin{proof} Note that quasi $(\phi,S)$-locality is a special case of quasi $S(x_1,x_2)$-locality in the sense of \cite{li-qvta}.
Then by Proposition 4.6 therein,  $U$ is a quasi compatible, so that we have a nonlocal vertex algebra $\<U\>_{\phi}$.
For $a(x),b(x)$, let $u_i(x),v_i(x)\in U,\  f_i(z)\in \C(z)$ $(i=1,\dots,r)$ and
$0\ne q(x_1,x_2)\in \C[x_1,x_2]$ such that (\ref{phi-S-locality}) holds.
By Theorem \ref{prop:tech-calculation5}, we have
\begin{align}\label{q-S-locality}
&q(\phi(x,x_1),\phi(x,x_2))Y_{\E}^{\phi}(a(x),x_1)Y_{\E}^{\phi}(b(x),x_2)\nonumber\\
=\ &q(\phi(x,x_1),\phi(x,x_2))\sum_{i=1}^r\iota_{x_2,x_1}f_i(\psi(x_1-x_2))Y_{\E}^{\phi}(u_i(x),x_2)Y_{\E}^{\phi}(v_i(x),x_1).
\end{align}
As $q(x_1,x_2)\ne 0$, we have $q(x_1,x_2)=(x_1-x_2)^{k}\bar{q}(x_1,x_2)$
for some $k\in \N,\  \bar{q}(x_1,x_2)\in \C[x_1,x_2]$
with $\bar{q}(x_2,x_2)\ne 0$. Noticing that $\phi(x,z)=x+p(x)z+\frac{1}{2}p(x)p'(x)z^2+\cdots$, we have
$$q(\phi(x,x_1),\phi(x,x_2))=(x_1-x_2)^{k}Q(x,x_1,x_2)$$
for some $Q(x,x_1,x_2)\in \C((x))[[x_1,x_2]]$ with $Q(x,0,0)\ne 0$.
Noticing that $Q(x,x_1,x_2)$ is invertible in $\C((x))[[x_1,x_2]]$,
then  from (\ref{q-S-locality}) by cancellation we get
\begin{align*}
&(x_1-x_2)^kY_{\E}^{\phi}(a(x),x_1)Y_{\E}^{\phi}(b(x),x_2)\nonumber\\
=\ &(x_1-x_2)^k\sum_{i=1}^r\iota_{x_2,x_1}f_i(\psi(x_1-x_2))Y_{\E}^{\phi}(u_i(x),x_2)Y_{\E}^{\phi}(v_i(x),x_1).
\end{align*}
This shows that $\{Y_{\E}^{\phi}(a(z),x)\ |\  a(z)\in U\}$ is an $S$-local set of vertex operators
on the nonlocal vertex algebra $\<U\>_{\phi}$.
Since $\<U\>_{\phi}$ as a nonlocal vertex algebra is generated by $U$,
from \cite{ltw} (Proposition 2.6), $\<U\>_{\phi}$ is a weak quantum vertex algebra.
\end{proof}

\section{$(G,\chi_{\phi})$-equivariant $\phi$-coordinated quasi modules}

In this section, we study $(G,\chi_{\phi})$-equivariant $\phi$-coordinated quasi modules
for a general $(G,\chi)$-module nonlocal vertex algebra. The main result
(Theorem \ref{coro:G-va-abs-construct-non-h-adic}) is
 a refinement of Theorem \ref{thm:abs-construct-non-h-adic}.
Throughout this section, $G$ is a group equipped with a linear character $\chi$.

First, we formulate the following notion (cf.  \cite{li-new}, \cite{li-gamma}):

\begin{de}\label{def-Gchiva}
Let $G$ be a group with  a linear character $\chi: G\rightarrow \C^{\times}$.
A {\em $(G,\chi)$-module nonlocal vertex algebra} is a nonlocal vertex algebra $V$ equipped with
a group homomorphism $R:G\rightarrow \te{GL} (V)$  such that
$R(g)\vac=\vac$ and
\begin{align}
  R(g)Y(v,x)R(g)\inverse =Y(R(g)v,\chi(g)x)\quad\te{for }g\in G,\  v\in V.
\end{align}
We shall also denote a $(G,\chi)$-module nonlocal vertex algebra by a pair $(V,R)$.
\end{de}

\begin{rem}
{\em Notice that in Definition \ref{def-Gchiva}, for $g\in G$ with $\chi(g)=1$,
we have $R(g)\in \te{Aut}(V)$, where $\te{Aut}(V)$ denotes the automorphism group of $V$
as a nonlocal vertex algebra. Thus a $(G,\chi)$-module nonlocal vertex algebra with $\chi=1$
(the trivial character) is simply
a nonlocal vertex algebra on which  $G$ acts as an automorphism group.
In this case, we call $V$ a {\em $G$-module nonlocal vertex algebra.}}
\end{rem}

 For $(G,\chi)$-module nonlocal vertex algebras $(V,R)$ and $(V',R')$,
a {\em $(G,\chi)$-module nonlocal vertex algebra homomorphism} is a nonlocal vertex algebra homomorphism
$\psi:V\rightarrow V'$ such that
\begin{align}
&  \psi\circ R(g)=R'(g)\circ \psi\quad\te{for }g\in G.
\end{align}

The following is a convenient technical result:

\begin{lem}\label{lem:G-va-construct}
Suppose that $V$ is a nonlocal vertex algebra,
$\rho:G\rightarrow \te{Aut}(V)$ and $L:G\rightarrow \te{GL} (V)$
are  group homomorphisms such that $L(g){\bf 1}={\bf 1}$,
\begin{align*}
&\rho(g)L(g')=L(g')\rho(g)  \quad\te{for }g,g'\in G,\\
&L(g)Y(v,x)L(g)\inverse=Y(L(g)v,\chi(g)x)  \quad\te{for }g\in G,\   v\in S,
\end{align*}
where $S$ is a generating subset of $V$.
Define a map $R:G\rightarrow \te{GL} (V)$ by $R(g)= \rho(g)L(g)$ for $g\in G$.
Then $(V,R)$ is a $(G,\chi)$-module nonlocal vertex algebra.
\end{lem}

\begin{proof} We first prove that $(V,L)$ is a $(G,\chi)$-module nonlocal vertex algebra.
For this we need to prove
\begin{align*}
  L(g)Y(v,x)L(g)\inverse=Y(L(g)v,\chi(g)x)
\end{align*}
for $g\in G,\ v\in V$.
Let $V'$  consist of those vectors $v\in V$ such that the above relation holds for all $g\in G$.
Then we must prove $V'=V$. As $S\cup \{ {\bf 1}\}\subset V'$ from assumption,
we need to prove that $V'$ is a nonlocal vertex subalgebra of $V$.
It remains to prove that $u_{n}v\in V'$ for $u,v\in V',\ n\in \Z$.
Let $u,v\in V'$ and let $g\in G$. There exists $k\in \N$ such that
\begin{align*}
&(x_1-x_2)^kY(u,x_1)Y(v,x_2)\in \te{Hom}(V,V((x_1,x_2))),\\
&(x_1-x_2)^kY(L(g)u,x_1)Y(L(g)v,x_2)\in \te{Hom}(V,V((x_1,x_2)))
\end{align*}
and
\begin{align*}
&z^kY(Y(u,z)v,x)= \te{Res}_{x_1}x_1\inverse\delta\(\frac{x+z}{x_1}\) \left((x_1-x)^kY(u,x_1)Y(v,x)\right),\\
&z^kY(Y(L(g)u,z)L(g)v,x)
= \te{Res}_{x_1}x_1\inverse\delta\(\frac{x+z}{x_1}\) \left((x_1-x)^kY(L(g)u,x_1)Y(L(g)v,x)\right).
\end{align*}
Then
\begin{align*}
&z^{k}L(g)Y(Y(u,z)v,x)L(g)\inverse\\
=\  &\te{Res}_{x_1}x_1\inverse\delta\(\frac{x+z}{x_1}\) \left((x_1-x)^kL(g)Y(u,x_1)Y(v,x)L(g)^{-1}\right)\\
=\  &\te{Res}_{x_1}x_1\inverse\delta\(\frac{x+z}{x_1}\) \left((x_1-x)^kY(L(g)u,\chi(g)x_1)Y(L(g)v,\chi(g)x)\right)\\
=\  &\chi(g)^{-k}\te{Res}_{\chi(g)x_1}
(\chi(g)x_1)\inverse\delta\(\frac{\chi(g)x+\chi(g)z}{\chi(g)x_1}\) \Big((\chi(g)x_1-\chi(g)x)^k\\
&\qquad\times
Y(L(g)u,\chi(g)x_1)Y(L(g)v,\chi(g)x)\Big)\\
=\  &\chi(g)^{-k}(\chi(g)z)^k
    Y\Big(Y(L(g)u,\chi(g)z)L(g)v,\chi(g)x\Big)\\
=\  &z^k Y\Big(L(g)Y(u,z)v,\chi(g)x\Big),
\end{align*}
which gives
$$L(g)Y(Y(u,z)v,x)L(g)\inverse=Y\Big(L(g)Y(u,z)v,\chi(g)x\Big).$$
This proves that $u_{n}v\in V'$ for $u,v\in V',\ n\in \Z$.
Thus $V'=V$. Therefore $(V,L)$ is a $(G,\chi)$-module nonlocal vertex algebra.
Furthermore, since $\rho$ is a group homomorphism from $G$ to the automorphism group $\te{Aut}(V)$,
it is straightforward to show that $(V,R)$ is a $(G,\chi)$-module nonlocal vertex algebra.
\end{proof}

For a subgroup $\Gamma$ of $\C^\times$, denote by $\C_{\Gamma}[x]$ the monoid
generated multiplicatively by polynomials $x-\alpha$ for $\alpha\in \Gamma$,
i.e., the set of all monic polynomials whose roots are contained in $\Gamma$.

\begin{de}\label{de:G-equiv-phi-mod}
Let $(V,R)$ be a $(G,\chi)$-module nonlocal vertex algebra and
 let $\chi_{\phi}: G\rightarrow \C^{\times}$ be a linear character of $G$.
A {\em $(G,\chi_{\phi})$-equivariant $\phi$-coordinated quasi $V$-module} is a $\phi$-coordinated quasi
$V$-module $(W,Y_W^\phi)$ satisfying the conditions that
\begin{align}\label{phi-module-equiv}
  Y_W^\phi\(R(g)v,x\)
  =Y_W^{\phi}(v,\chi_{\phi}(g)^{-1}x)\quad\te{for }g\in G,\  v\in V
\end{align}
and that for any $u,v\in V$, there exists $f(x)\in\C_{\chi_{\phi}(G)}[x]$
such that
\begin{align}\label{eq:quasi-compatible-temp-tt01}
  f(x_1/x_2)Y_W^\phi(u,x_1)Y_W^\phi(v,x_2)\in\te{Hom} (W,W((x_1,x_2))).
\end{align}
\end{de}

\begin{lem}\label{equivariance-character-compat}
Let $(V,R)$ be a $(G,\chi)$-module nonlocal vertex algebra and let $\chi_{\phi}: G\rightarrow \C^{\times}$ be a linear character of $G$.
Suppose that there exists a $(G,\chi_{\phi})$-equivariant $\phi$-coordinated quasi $V$-module
$(W,Y_{W}^{\phi})$ such that $\frac{d}{dx}Y_{W}^{\phi}(v,x)\ne 0$ for some $v\in V$. Then
\begin{align}\label{phi-chi-chiphi}
\phi(x,\chi(g)x_0)=\chi_{\phi}(g)\phi (\chi_{\phi}(g)^{-1}x,x_0)\   \   \   \mbox{ for }g\in G.
\end{align}
On the other hand,  (\ref{phi-chi-chiphi}) is equivalent to
\begin{align}\label{def-chi}
p(\chi_{\phi}(g)x)=\chi(g)^{-1}\chi_{\phi}(g)p(x)\   \   \mbox{ for }g\in G,
\end{align}
where  $\phi(x,z)=e^{zp(x)\partial_x}x$ with $p(x)\in \C((x))$.
\end{lem}

\begin{proof} We first prove the second assertion. As $\phi(x,z)=e^{zp(x)\partial_x}x$, we have
\begin{eqnarray*}
&&\phi(x,\alpha x_0)=e^{\alpha x_0p(x)\partial_{x}}x,\\
&&\beta \phi(\beta^{-1}x,x_0)=\beta e^{x_0p(\beta^{-1}x)\partial_{\beta^{-1}x}}
(\beta^{-1}x)=e^{x_0\beta p(\beta^{-1}x)\partial_{x}}x
\end{eqnarray*}
 for any $\alpha,\beta\in \C^{\times}$. It then follows immediately that the second assertion holds.

To prove the first assertion, let $v\in V$ such that $\frac{d}{dx}Y_{W}^{\phi}(v,x)\ne 0$.
Recall that $Y_{W}^{\phi}({\bf 1},x)=1$ and $R(g){\bf 1}={\bf 1}$ for $g\in G$.
Notice that for $a,b\in V$, if $h(x_1,x_2)$ is any nonzero polynomial such that
$$h(x_1,x_2)Y_{W}^{\phi}(a,x_1)Y_{W}^{\phi}(b,x_2)\in {\rm Hom} (W,W((x_1,x_2))),$$
 then
$$\left(h(x_1,x_2)Y_{W}^{\phi}(a,x_1)Y_{W}^{\phi}(b,x_2)\right)|_{x_1=\phi(x_2,x_0)}
=h(\phi(x_2,x_0),x_2)Y_{W}^{\phi}(Y(a,x_0)b,x_2).$$
Using this fact and  (\ref{phi-module-equiv}), for $g\in G$ we get
\begin{eqnarray*}
&&\left(Y_{W}^{\phi}(v,x_1)Y_{W}^{\phi}({\bf 1},x_2)\right)|_{x_1=\phi(x_2,\chi(g)x_0)}\nonumber\\
&=&Y_{W}^{\phi}(Y(v,\chi(g)x_0){\bf 1},x_2)\nonumber\\
&=&Y_{W}^{\phi}(R(g)Y(R(g)^{-1}v,x_0){\bf 1},x_2)\nonumber\\
&=&Y_{W}^{\phi}(Y(R(g)^{-1}v,x_0){\bf 1},\chi_{\phi}(g)^{-1}x_2)\nonumber\\
&=&\left(Y_{W}^{\phi}(R(g)^{-1}v,x_1)Y_{W}^{\phi}({\bf 1},\chi_{\phi}(g)^{-1}x_2)\right)|_{x_1=\phi(\chi_{\phi}(g)^{-1}x_2,x_0)}\nonumber\\
&=&\left(Y_{W}^{\phi}(v,\chi_{\phi}(g)x_1)Y_{W}^{\phi}({\bf 1},x_2)\right)|_{x_1=\phi(\chi_{\phi}(g)^{-1}x_2,x_0)}\nonumber\\
&=&\left(Y_{W}^{\phi}(v,x_1)Y_{W}^{\phi}({\bf 1},x_2)\right)|_{x_1=\chi_{\phi}(g)\phi(\chi_{\phi}(g)^{-1}x_2,x_0)}.
\end{eqnarray*}
As $\frac{d}{dx}Y_{W}^{\phi}(v,x)\ne 0$, there exist $w\in W,\ \alpha\in W^{*}$ such that
$\frac{d}{dx}\langle \alpha, Y_{W}^{\phi}(v,x)w\rangle\ne 0$. Set $F(x)=\langle \alpha, Y_{W}^{\phi}(v,x)w\rangle\in \C((x))$.
Then we have $F'(x)\ne 0$ and
$$F(\phi(x_2,\chi(g)x_0))=F(\chi_{\phi}(g)\phi(\chi_{\phi}(g)^{-1}x_2,x_0)).$$
As
\begin{eqnarray*}
&&F(\phi(x_2,\chi(g)x_0))=e^{\chi(g)x_0p(x)\partial_{x_2}} F(x_2),\\
&&F\left(\chi_{\phi}(g)\phi(\chi_{\phi}(g)^{-1}x_2,x_0)\right)
=e^{x_0\chi_{\phi}(g) p(\chi_{\phi}(g)^{-1}x)\partial_{x_2}}F(x_2),
\end{eqnarray*}
 by extracting the coefficients of $x_0$ we obtain (\ref{phi-chi-chiphi}).
\end{proof}

\begin{rem}\label{trivial-case}
{\em Under the setting of Lemma \ref{equivariance-character-compat},
assume that $(W,Y_{W}^{\phi})$ is a $(G,\chi_{\phi})$-equivariant $\phi$-coordinated quasi $V$-module
 such that $\frac{d}{dx}Y_{W}^{\phi}(v,x)=0$ for all $v\in V$.
It is straightforward to show that the quotient nonlocal vertex algebra
$A:=V/(\ker Y_{W}^{\phi})$ is simply an associative algebra on which $G$ acts trivially
and $W$ is a (faithful) $A$-module.}
\end{rem}

In view of Lemma \ref{equivariance-character-compat} and Remark \ref{trivial-case},
from now on {\em we shall always assume that (\ref{phi-chi-chiphi})
holds for $(G,\chi_{\phi})$-equivariant $\phi$-coordinated quasi $V$-modules} unless it is stated otherwise.

\begin{rem}\label{lem:compatible-associate}
{\em Let $\Gamma$ be a subgroup of $\C^{\times}$ and let $p(x)\in \C((x))$ nonzero.
It is straightforward to show that there exists a linear character
$\lambda: \Gamma\rightarrow \C^{\times}$ such that
\begin{align}
p(cx)=\lambda(c)  p(x) \   \   \   \mbox{ for }c\in \Gamma
\end{align}
if and only if $p(x)=x^r$ for some $r\in \Z$ when $|\Gamma|=\infty$
and
$p(x)=x^rp_{0}(x^{m})$ for some $r\in \Z, \ p_0(x)\in \C[[x]]$ with $p_0(0)\ne 0$ when $|\Gamma|=m<\infty$.}
\end{rem}

\begin{rem}
{\em  Consider the case $\phi(x,z)=x+z$ with $p(x)=1$. Then the condition (\ref{def-chi}) amounts to  $\chi=\chi_{\phi}$.
In this case, the notion of $(G,\chi_{\phi})$-equivariant $\phi$-coordinated
quasi $V$-module coincides with the notion of  $G$-equivariant quasi $V$-module introduced in \cite{li-new}.
On the other hand, suppose $\phi(x,z)=xe^z$ with $p(x)=x$. From (\ref{def-chi}), we have $\chi=1$, so that
$G$ acts on $V$ as an automorphism group.
In this case, $\chi_{\phi}$ can be any linear character  of $G$ and
the notion of $(G,\chi_{\phi})$-equivariant $\phi$-coordinated quasi $V$-module coincides with the
same named notion introduced in \cite{li-jmp} with $\chi=\chi_{\phi}^{-1}$.}
\end{rem}

The following technical result is analogous to Lemma \ref{lem:G-va-construct}:

\begin{lem}\label{lem:G-equiv-phi-mod-construct}
Under the setting of Definition \ref{de:G-equiv-phi-mod}, assume that all the conditions hold except
(\ref{phi-module-equiv}), and instead assume
\begin{align}
  Y_W^\phi(R(g)v,x)=Y_W^\phi(v,\chi_{\phi}(g)^{-1}x)\quad\te{for }g\in G,\  v\in S,
\end{align}
where $S$ is a generating subset of $V$.
Then $W$ is a $(G,\chi_{\phi})$-equivariant $\phi$-coordinated quasi $V$-module.
\end{lem}

\begin{proof}
Let $V'$ consist of all vectors $v\in V$ such that
\begin{align*}
  Y_W^\phi(R(g)v,x)=Y_W^\phi(v,\chi_{\phi}(g)^{-1}x)\quad\te{for }g\in G.
\end{align*}
Let $u,v\in V'$ and let $g\in G$.
Notice that there exists $f(x)\in\C_{\chi_{\phi}(G)}[x]$ such that
\begin{align*}
f(x_1/x)Y_W^\phi(u,x_1)Y_W^\phi(v,x),\ f(x_1/x)Y_W^\phi(R(g)u,x_1)Y_W^\phi(R(g)v,x)
\in \te{Hom}(W,W((x_1,x_2))),
\end{align*}
and
\begin{align*}
f(\phi(x,z)/x)Y_W^\phi(Y(u,z)v,x)= \te{Res}_{x_1}x_1\inverse\delta\(\frac{\phi(x,z)}{x_1}\)
\left(f(x_1/x)Y_W^\phi(u,x_1)Y_W^\phi(v,x)\right),
\end{align*}
\begin{align*}
&f(\phi(x,z)/x)Y_W^\phi(Y(R(g)u,z)R(g)v,x)\\
=& \te{Res}_{x_1}x_1\inverse\delta\(\frac{\phi(x,z)}{x_1}\)
\left(f(x_1/x)Y_W^\phi(R(g)u,x_1)Y_W^\phi(R(g)v,x)\right).
\end{align*}
Using all of these and  (\ref{phi-chi-chiphi}) we get
\begin{align*}
&f(\phi(x,\chi(g)z)/x)Y_W^\phi(R(g)Y(u,z)v,x)\\
=&f(\phi(x,\chi(g)z)/x)Y_W^\phi(Y(R(g)u,\chi(g)z)R(g)v,x)\\
=&\te{Res}_{x_1}x_1\inverse\delta\(\frac{\phi(x,\chi(g)z)}{x_1}\)
\left(f(x_1/x)Y_W^\phi(R(g)u,x_1)Y_W^\phi(R(g)v,x)\right)\\
=&\te{Res}_{x_1}x_1\inverse\delta\(\frac{\phi(x,\chi(g)z)}{x_1}\)
\left(f(x_1/x)Y_W^\phi(u,\chi_{\phi}(g)^{-1}x_1)Y_W^\phi(v,\chi_{\phi}(g)^{-1}x)\right)\\
=&\te{Res}_{x_1}x_1\inverse\delta\(\frac{\phi(\chi_{\phi}(g)^{-1}x,z)}{\chi_{\phi}(g)^{-1}x_1}\)
\left(f(x_1/x)Y_W^\phi(u,\chi_{\phi}(g)^{-1}x_1)Y_W^\phi(v,\chi_{\phi}(g)^{-1}x)\right)\\
=&f(\phi(\chi_{\phi}(g)^{-1}x,z)/\chi_{\phi}(g)^{-1}x)Y_W^\phi(Y(u,z)v,\chi_{\phi}(g)^{-1}x)\\
=&f(\phi(x,\chi(g)z)/x)Y_W^\phi(Y(u,z)v,\chi_{\phi}(g)^{-1}x).
\end{align*}
Multiplying both sides by $f(\phi(x,\chi(g)z)/x)\inverse$ in $\C((x))((z))$,
we get
\begin{align*}
Y_W^\phi(R(g)Y(u,z)v,x)=
Y_W^\phi(Y(u,z)v,\chi_{\phi}(g)^{-1}x).
\end{align*}
This proves that $u_{n}v\in V'$ for $u,v\in V',\ n\in \Z$.
Hence, $V'$ is a nonlocal vertex subalgebra of $V$. As $S\subset V'$ by assumption, we must have $V'=V$.
Therefore, $W$ is a $(G,\chi_{\phi})$-equivariant $\phi$-coordinated quasi $V$-module.
\end{proof}

In the following, we shall get a refinement of Theorem \ref{thm:abs-construct-non-h-adic} in terms of
$(G,\chi)$-module nonlocal vertex algebras and $(G,\chi_{\phi})$-equivariant $\phi$-coordinated quasi modules.
Let $W$ be a vector space and let $G$ be a subgroup of $\C^\times$.
For $g\in G$, define a linear automorphism $R_g$ of $\E(W)$ by
\begin{align}\label{def-nug}
R_g(a(x))=a(g^{-1} x)\    \    \    \mbox{  for }a(x)\in \E(W).
\end{align}
This gives a group homomorphism  $R$ from $G$ to $\te{GL}(\E(W))$.

Let $\phi(x,z)=\exp(zp(x)\frac{d}{dx})x$ with
\begin{eqnarray}\label{special-p(x)}
p(x)=\begin{cases}x^{r+1}  &\mbox{  if }|G|=\infty\\
x^{r+1} p_0(x^m)&  \mbox{  if } |G|=m,
\end{cases}
\end{eqnarray}
where $r\in \Z$, $p_0(x)\in \C[[x]]$ such that $p_0(0)\ne 0$.
Take $\chi_{\phi}$ to be the natural embedding of $G$ into $\C^{\times}$, i.e.,
\begin{eqnarray}
\chi_{\phi}(g)=g\   \   \    \mbox{ for }g\in G.
\end{eqnarray}
Then using $\chi_{\phi}$ we define a linear character $\chi:  G\rightarrow \C^{\times}$ by
(\ref{def-chi}), i.e.,
\begin{align}
 \chi(g)^{-1}=\frac{p(\chi_{\phi}(g)x)}{\chi_{\phi}(g)p(x)}=\frac{p(gx)}{gp(x)}\   \   \   \mbox{ for }g\in G.
\end{align}
More explicitly, we have
\begin{align}
 \chi(g)=g^{-r}\   \   \   \mbox{ for }g\in G.
\end{align}
We also have
\begin{align}\label{phi-gx}
\phi(g^{-1}x,x_0)=g^{-1}\phi(x,\chi(g)x_0)\   \   \   \mbox{ for }g\in G.
\end{align}

\begin{de}
A subset $U$ of $\E(W)$ is said to be {\em $G$-quasi compatible} if
 for any finite sequence $a_1(x),a_2(x),\dots,a_k(x)\in U$, there exists
$ f(x)\in \C_{G}[x]$ such that
\begin{align*}
  \left(\prod_{1\le i<j\le k}f(x_i/x_j)\right)a_1(x_1)a_2(x_2)\cdots a_k(x_k)\in\te{Hom}(W,W((x_1,\dots,x_k))).
\end{align*}
\end{de}

From definition, it is clear that any $G$-quasi compatible subset of $\E(W)$ is quasi compatible.
Let $U$ be a $G$-quasi compatible subset of $\E(W)$.
In view of Theorem \ref{thm:abs-construct-non-h-adic},
$U$ generates a nonlocal vertex algebra $\<U\>_\phi$ with $W$ as a $\phi$-coordinated quasi module.
The following is a refinement of Theorem \ref{thm:abs-construct-non-h-adic}:

\begin{thm}\label{coro:G-va-abs-construct-non-h-adic}
Let $W$ be a vector space and let $G$ be a subgroup of $\C^\times$.
Assume that $U$ is a $G$-quasi compatible subset of $\E(W)$, which is $G$-stable.
Then the nonlocal vertex algebra $\<U\>_\phi$ together with the group homomorphism $R$
defined in (\ref{def-nug}) is a $(G,\chi)$-module nonlocal vertex algebra
and $W$ is a $(G,\chi_{\phi})$-equivariant $\phi$-coordinated quasi $\<U\>_\phi$-module with
$Y_W(a(x),z)=a(z)$ for $a(x)\in \<U\>_\phi$.
\end{thm}

\begin{proof}
Let $(a(x),b(x))$ be any quasi-compatible pair in $\E(W)$. Then for any $g,h\in G$,
$(a(gx),b(hx))$ is a quasi-compatible pair in $\E(W)$. More specifically,
assume that $f(x)$ is a nonzero polynomial such that
$$f(x_1/x_2)a(x_1)b(x_2)\in \te{Hom}(W,W((x_1,x_2))).$$
Then
$$f(gh^{-1}x_1/x_2)a(gx_1)b(hx_2)\in \te{Hom}(W,W((x_1,x_2))).$$
From definition we have
\begin{align}\label{3.13-first}
\left(f(x_1/g^{-1}x)a(x_1)b(g^{-1}x)\right)|_{x_1=\phi(g^{-1}x,x_0)}
=f(\phi(g^{-1}x,x_0)/g^{-1}x)R_g \left(Y_{\E}^{\phi}(a(x),x_0)b(x)\right).
\end{align}
Using (\ref{phi-gx}) and (\ref{3.13-first}) we get
\begin{align*}
&f(\phi(x,\chi(g)x_0)/x)R_g \left(Y_{\E}^{\phi}(a(x),x_0)b(x)\right)\\
=&\ f(\phi(g^{-1}x,x_0)/g^{-1}x)R_g \left(Y_{\E}^{\phi}(a(x),x_0)b(x)\right)\\
=& \left(f(x_1/g^{-1}x)a(x_1)b(g^{-1}x)\right)|_{x_1=g^{-1}\phi(x,\chi(g)x_0)}\\
=&\left(f(x_1/x)a(g^{-1}x_1)b(g^{-1}x)\right)|_{x_1=\phi(x,\chi(g)x_0)}\\
=&\ f(\phi(x,\chi(g)x_0)/x)Y_{\E}^{\phi}\left(R_{g}(a(x)),\chi(g)x_0\right)R_{g}(b(x)),
\end{align*}
which implies
\begin{align}
&R_g \left(Y_{\E}^{\phi}(a(x),x_0)b(x)\right)=Y_{\E}^{\phi}\left(R_{g}(a(x)),\chi(g)x_0\right)R_{g}(b(x)).
\label{quasi-autom}
\end{align}
As $U$ is $G$-stable and $\<U\>_{\phi}$ as a nonlocal vertex algebra is generated by $U$,
 it follows from (\ref{quasi-autom}) that $R_g$ preserves $\<U\>_\phi$ for every $g\in G$.
It also follows  from (\ref{quasi-autom}) that the nonlocal vertex algebra $\<U\>_\phi$ together with the group homomorphism $R$
is a $(G,\chi)$-module nonlocal vertex algebra.

From Theorem \ref{thm:abs-construct-non-h-adic}, $W$ is naturally a $\phi$-coordinated quasi $\<U\>_\phi$-module.
For $a(x)\in \<U\>_\phi, \ g\in G$, we have
 $$Y_{W}(R_{g}a(x),z)=Y_{W}(a(g^{-1}x),z)=a(g^{-1}z)=Y_{W}(a(x),g^{-1}z)=Y_{W}(a(x),\chi_{\phi}(g)^{-1}z).$$
On the other hand, following the proof of \cite[Proposition 4.9]{li-cmp}, we see that $\<U\>_\phi$ is also $G$-quasi compatible.
Therefore, $W$ is a $(G,\chi_{\phi})$-equivariant $\phi$-coordinated quasi $\<U\>_\phi$-module.
\end{proof}

\section{$\phi$-coordinated quasi modules for lattice vertex algebras}\label{sec:lattice-vertex-algebras}

In this section, we shall use the general results in Section 3 and a result of Lepowsky
to construct $\phi$-coordinated quasi modules for lattice vertex algebras.

\subsection{Lattice vertex algebras}
In this subsection, we recall the construction of vertex algebras associated to non-degenerate even lattices
and some results involving the associative algebra $A(L)$ introduced in \cite{LL}.

We start by briefly recalling the construction of lattice vertex algebras.
(For a detailed exposition, see \cite[Section 6.4--6.5]{LL} for example.)
Let $L$ be a finite rank  even lattice  in the sense that $L$ is a free abelian group of finite rank equipped with a
symmetric $\Z$-valued bilinear form $\<\cdot,\cdot\>$ such that $\<\alpha,\alpha\>\in 2\Z$ for $\alpha\in L$.
We assume that $L$ is nondegenerate in the obvious sense.
Set $$\h=\C\ot_{\Z}L$$
 and extend $\<\cdot,\cdot\>$ to a symmetric $\C$-valued bilinear form on $\h$.
View $\h$ as an abelian Lie algebra with $\<\cdot,\cdot\>$ as a nondegenerate symmetric invariant bilinear form.
Then we have an affine Lie algebra $\wh \h$. By definition,
$$\wh \h=\h\otimes \C[t,t^{-1}]+\C {\bf k}$$
as a vector space, where ${\bf k}$ is central and
\begin{eqnarray}
[\alpha(m),\beta(n)]=m\delta_{m+n,0}\<\alpha,\beta\>{\bf k}
\end{eqnarray}
for $\alpha,\beta\in \h,\ m,n\in \Z$ with $\alpha(m)$ denoting $\alpha\otimes t^{m}$.  Set
\begin{eqnarray}
\wh \h^{\pm}=\h\otimes t^{\pm 1} \C[t^{\pm 1}],
\end{eqnarray}
which are abelian Lie subalgebras. Identify $\h$ with $\h\otimes t^{0}$.
Notice that the center of $\wh \h$ equals $\h+\C {\bf k}$.
Furthermore, set
\begin{eqnarray}
\wh \h'=\wh \h^{+}+\wh \h^{-}+\C {\bf k},
\end{eqnarray}
which is a Heisenberg algebra. Then $\wh \h=\wh \h'\oplus \h$, which is a direct sum of Lie algebras.

Denote by $L^{o}$  the dual lattice of $L$, i.e.,
\begin{eqnarray}
L^{o}=\set{\al\in\h }{\<\al,\be\>\in \Z\  \mbox{ for }\beta\in L}.
\end{eqnarray}
Fix a positive integer $s$ such that
$$s\<\alpha, \beta\>\in \Z\   \   \   \mbox{ for }\alpha,\beta\in L^o.$$
Furthermore, set $\omega=e^{\pi i/s}$, the principal primitive $2s$-th  root of unity.

Let $\epsilon:L^o\times L^o\rightarrow \C^\times$ be a 2-cocycle such that
\begin{eqnarray*}
&&\epsilon(\alpha,\beta)\epsilon(\beta,\alpha)^{-1}=\omega^{s\<\al,\be\>},\\
&&\epsilon(\alpha,0)=1=\epsilon(0,\alpha)
\end{eqnarray*}
 for $\alpha,\beta \in L^o$.
 Such a $2$-cocycle $\epsilon$ indeed exists (cf. Remark 6.4.1, \cite{LL}).
Note that $\omega^{s\<\al,\be\>}=(-1)^{\<\al,\be\>}$ if $\<\al,\be\>\in \Z$.

Let $\C_{\epsilon}[L^o]$ be the $\epsilon$-twisted group algebra of $L^{o}$,
which by definition has a designated basis
$\{ e_{\alpha}\ |\ \alpha\in L^o\}$
with $e_\al \cdot e_\be=\epsilon\(\al,\be\)e_{\al+\be}$ for $\al,\be\in L^o$.
We make $\C_{\epsilon}[L^o]$ an $\wh \h$-module by letting $\wh \h'$
act trivially and letting $\h$ act by
\begin{eqnarray}
  he_\be=\<h,\be\>e_\be \   \  \mbox{ for }h\in \h,\  \be\in L^o.
  \end{eqnarray}
For $\al\in L$,  define a linear operator $z^{\al}:\   \C_{\epsilon}[L^o]\rightarrow \C_{\epsilon}[L^o][z,z^{-1}]$ by
\begin{eqnarray}
z^\al\cdot e_\be=z^{\<\al,\be\>}e_\be\   \   \   \mbox{ for }\be\in L^{o}.
\end{eqnarray}

Recall that the canonical irreducible module for the Heisenberg algebra $\wh \h'$ of level $1$
is isomorphic to $S(\wh \h^{-})$ as an $\wh \h^{-}$-module.
Then we view $S(\wh \h^{-})$ as an $\widehat \h$-module of level $1$ with $\h$ acting trivially.
Set
\begin{eqnarray}
V_{L^o}=S(\wh \h^{-})\otimes \C_{\epsilon}[L^o],
\end{eqnarray}
the tensor product of $\widehat \h$-modules. Then $V_{L^o}$ is an $\widehat \h$-module of level $1$.
More generally, for any subset $K$ of $L^o$, we define
$$V_K=\sum_{\al\in K}S(\wh \h^{-})\otimes \C e_{\al}\subset V_{L^o},$$
which  is an $\widehat \h$-submodule of $V_{L^o}$.

Set
$$\vac=e_0\in \C_{\epsilon}[L^o]\subset V_{L^o}.$$
Identify $\h$ and $\C_{\epsilon}[L^o]$ as subspaces of $V_{L^o}$ via $a\mapsto a(-1)\otimes 1$ ($a\in\h$)
and $e_\al\mapsto 1\otimes e_\al$ ($\al\in L^o$).
For $\alpha\in \h$, set
\begin{align}
\al(z)=\sum_{n\in\Z}\al(n)z^{-n-1}\in \E(V_{L^o}).
\end{align}
For any $a(x),b(x)\in  \E(V_{L^o})$, define normally ordered product
\begin{align}
\nord a(x)b(x)  \nord=a(x)^{+}b(x)+b(x)a(x)^{-},
\end{align}
where $a(x)^{+}=\sum_{n<0}a_nx^{-n-1}$ and $a(x)^{-}=\sum_{n\ge 0}a_nx^{-n-1}$.
Define a linear map
\begin{eqnarray}
Y(\cdot,z):\  V_L\rightarrow (\te{End} V_{L^o})[[z,z^{-1}]]
\end{eqnarray}
by
\begin{align*}
&Y(\al,z)=\al(z),\\
&Y(e_\al,z)=\exp\(\sum_{n>0}\frac{\al(-n)}{n}z^n\)\exp\(-\sum_{n>0}\frac{\al(n)}{n}z^{-n}\)e_\al z^\al,\\
&Y(\al^{(1)}_{-n_1-1}\cdots \al^{(r)}_{-n_r-1}e_\al,z)
=\nord \partial_z^{(n_1)}\al^{(1)}(z)\cdots \partial_z^{(n_r)}\al^{(r)}(z)Y(e_\al,z)\nord,
\end{align*}
where $r\ge 0$, $n_1,\dots,n_r\ge 0$ and $\al^{(1)},\dots,\al^{(r)},\al\in L$.
Then  $\(V_L,Y,\vac\)$ carries the structure of a vertex algebra (see \cite{bor}, \cite{FLM}).
Set
\begin{align}\label{conformal-vector}
\omega=\sum_{i=1}^{d}u^{(i)}(-1)u^{(i)}(-1){\bf 1}\in V_{L},
\end{align}
where $u^{(1)},\dots,u^{(d)}$ is an orthonormal basis of $\h$, and
furthermore, set
$$Y(\omega,x)=\sum_{n\in \Z}L(n)x^{-n-1}.$$
Then
\begin{align}
[L(m),L(n)]=(m-n)L(m+n)+\frac{d}{12}(m^3-m)\delta_{m+n,0}
\end{align}
for $m,n\in \Z$, $V_{L}$ is $\Z$-graded by the eigenvalues of $L(0)$, called the {\em conformal weights},
and
$$Y(L(-1)v,x)=\frac{d}{dx}Y(v,x)\   \   \   \mbox{ for } v\in V_{L}.$$
Especially, we have
\begin{align}
L(0)h=h,\   \    \  \  L(0)e_{\alpha}=\frac{1}{2}\<\al,\al\>e_{\al}
\end{align}
for $h\in \h,\ \al\in L$.
On the other hand,  $V_{L^o}$ is a $V_{L}$-module and
for any $\lambda\in L^o$, $V_{\lambda+L}$ is an irreducible $V_L$-module.

Furthermore, we have (see \cite{dong1}, \cite[Theorem 3.16]{DLM}):

\begin{prop}\label{prop:va-lattice-mod-regularity}
Let $L$ be a nondegenerate even lattice. Then every $V_L$-module is completely reducible
and any irreducible $V_L$-module is isomorphic to $V_{\lambda+L}$ for some $\lambda\in L^o$.
Furthermore, for each $\al\in L$, $\al(0)$ acts semi-simply  with only integer eigenvalues on every $V_L$-module.
\end{prop}

Now, we recall the associative algebra $A(L)$ associated to $V_L$ (see \cite[Section 6.5]{LL}).
By definition, $A(L)$ is the associative algebra with unit $1$ generated by $\set{h[n],\  e_\al[n]}{h\in\h,\ \al\in L,\ n\in \Z}$,
subject to a set of relations written in terms of generating functions:
\begin{align*}
  h[z]=\sum_{n\in\Z}h[n]z^{-n-1},\,\,\,\,
  e_\al[z]=\sum_{n\in\Z}e_\al[n]z^{-n-1}.
\end{align*}
The relations are
\begin{align*}
  &\te{(AL1) }\  \   \   \   e_0[z]=1,\\
  &\te{(AL2) }\   \   \   \    \left[h[z],
    h'[w]\right]=\<h,h'\>\frac{\partial}{\partial w}z\inverse\delta\(\frac{w}{z}\),\\
  &\te{(AL3) }\    \   \  \    \left[h[z],e_\al[w]\right]=\<\al,h\>
    e_\al[w]z\inverse\delta\(\frac{w}{z}\),\\
  &\te{(AL4) }\    \   \   \   \left[e_\al[z],e_\be[w]\right]=0\   \te{ if }\<\al,\be\>\ge0,\\
  &\te{(AL5) }\   \   \   \  (z-w)^{-\<\al,\be\>}\left[e_\al[z],e_\be[w]\right]=0\   \te{ if }\<\al,\be\><0,
\end{align*}
for $h, h'\in \h,\  \al,\be\in L$.
An $A(L)$-module $W$ is said to be {\em restricted} if for every $w\in W$, we have
$ h[z]w, \  e_\al[z]w\in W((z))$ for all $h\in\h,\  \al\in L$.

The following result gives a connection
between $V_L$-modules and certain $A(L)$-modules (see \cite[Section 6.5]{LL}):

\begin{prop}\label{prop:VQ-mod-AQ-mod}
Let $(W,Y_W)$ be any $V_L$-module. Then $W$ is a restricted $A(L)$-module with
\begin{align*}
  h[z]=Y_W(h,z),\   \  \   \,\,e_\al[z]=Y_W(e_\al,z)
\end{align*}
for  $h\in\h,\   \al\in L$. Furthermore, the following relations hold on $W$ for $\al,\be\in L$:
\begin{align*}
&\te{(AL6) }\   \   \   \   \partial_z e_\al[z]=\nord \al[z] e_\al[z]\nord,\\
  &\te{(AL7) }\   \   \    \   \te{Res}_{x}\big(
    (x-z)^{-\<\al,\be\>-1} e_\al[x] e_\be[z]-(-z+x)^{-\<\al,\be\>-1} e_\be[z] e_\al[x]\big)\\
  &\hspace{2cm}=\varepsilon(\al,\be) e_{\al+\be}[z].
\end{align*}
On the other hand, let $W$ be a restricted $A(L)$-module
such that (AL6) and (AL7) hold.
Then $W$ admits a  $V_L$-module structure which is uniquely determined by
$$Y_W(h,z)= h[z],\   \   \    Y_W(e_\al,z)= e_\al[z]\   \  \   \mbox{ for }h\in\h,\   \al\in L,$$
where the module structure is explicitly given by
\begin{align*}
  Y_W(h^{(1)}_{-n_1-1}\cdots h^{(r)}_{-n_r-1}e_\al,z)
  =\nord \partial_z^{(n_1)} h^{(1)}[z]\cdots \partial_z^{(n_r)} h^{(r)}[z] e_\al[z]\nord
\end{align*}
for $r\ge 0$, $h^{(i)}\in\h$, $n_i\ge 0$ and $\al\in L$.
\end{prop}

On the other hand, we have the following universal property of $V_{L}$:

\begin{prop}\label{prop:AQ-mod-hom-VA-hom}
Let $V$ be a nonlocal vertex algebra
and let $\psi:\h\oplus\C_{\varepsilon}[L]\rightarrow V$ be a linear map such that
$\psi(e_0)={\bf 1}$, the relations (AL1-3) and (AL6) hold with
\begin{align*}
  h[z]= Y(\psi(h),z),\quad e_\al[z]= Y(\psi(e_\al),z)\quad
  \te{for }h\in\h,\  \al\in L,
\end{align*}
and such that the following relation holds for $\al,\be\in L$:
\begin{align}\label{A(L)457}
&(x-z)^{-\<\al,\be\>-1}e_{\alpha}[x]e_{\be}[z]
-(-z+x)^{-\<\al,\be\>-1}e_{\be}[z]e_{\alpha}[x]\nonumber\\
&\quad \quad \quad =  \varepsilon(\al,\be)e_{\al+\be}[z]x^{-1}\delta\left(\frac{z}{x}\right).
\end{align}
Then $\psi$ can be  extended uniquely to a nonlocal vertex algebra homomorphism from $V_L$ to $V$.
\end{prop}

\begin{proof}
From Proposition \ref{prop:VQ-mod-AQ-mod}, $V_L$ is a restricted $A(L)$-module with
$h[z]= Y(h,z)$ and $e_\al[z]= Y(e_\al,z)$ for $h\in\h,\  \al\in L$.
On the other hand, we deduce from relation \eqref{A(L)457} that (AL4), (AL5) and (AL7) holds on $V$.
Then $V$ is an $A(L)$-module such that (AL6) and (AL7) hold.
For $h\in\h$, we have
\begin{align*}
  h[z] \vac_V=Y_V(\psi^0(h),z)\vac_V\in V[[z]],
\end{align*}
which implies $h[n]\vac_V=0$ for $n\ge 0$.
It follows from \cite[Proposition 6.5.17]{LL} that
$\vac\mapsto \vac_V$ induces an $A(L)$-module isomorphism $\psi$ from $V_L$ to $A(L)\vac_V$.
Notice that
\begin{align*}
  &\psi(h)=\psi(h[-1]\vac)=h[-1]\psi(\vac)
  =\psi^0(h)_{-1}{\bf 1}_{V}=\psi^0(h),\\
  &\psi(e_\al)=\psi(e_\al[-1]\vac)=e_\al[-1]\psi(\vac)
  =\psi^0(e_\al)_{-1}\vac_V=\psi^0(e_\al)
\end{align*}
for $h\in\h,\  \al\in L$. Thus $\psi$ extends $\psi^0$ and we have
\begin{align*}
  &\psi(Y(u,x)v)=\psi(u[x]v)=Y_V(\psi(u),x)\psi(v)\   \   \   \mbox{ for }u\in\h\oplus\C_{\epsilon}[L],\  v\in V_L.
\end{align*}
As $\h+\C_{\epsilon}[L]$ generates $V_L$ as a vertex algebra, it follows that $\psi$
is a nonlocal vertex algebra homomorphism.
Therefore, $\psi$ is an injective  nonlocal vertex algebra homomorphism,
which extends $\psi^0$. The uniqueness is clear.
\end{proof}

\subsection{$\phi$-coordinated quasi-modules for $V_L$}

In this subsection, we first generalize the twisted vertex operators constructed in \cite{Lep} and
then construct $\phi$-coordinated quasi-modules for $V_L$.

Let $\mu$ be an isometry of the lattice $L$ with a period $N$, which are fixed throughout this subsection.
As in \cite{Lep}, assume
\begin{align}\label{sum-even}
  \sum_{k\in\Z_N}\<\mu^k(\al),\al\>\in 2\Z\quad\te{for }\al\in L.
\end{align}
Recall $\h=\C\otimes_{\Z}L$. We then view $\mu$ naturally
as a linear automorphism of $\h$, so that $\<\mu(h),\mu(h')\>=\<h,h'\>$ for $h,h'\in\h$.

Notice that for $\al,\be\in L$, we have
\begin{align*}
  \epsilon(\mu(\al),\mu(\be))\epsilon(\mu(\be),\mu(\al))\inverse
  =(-1)^{\<\mu(\al),\mu(\be)\>}=(-1)^{\<\al,\be\>}
  =\epsilon(\al,\be)\epsilon(\be,\al)\inverse.
\end{align*}
From \cite[Proposition 5.4.1]{FLM}, $\mu$ can be lifted to an automorphism of
$\C_{\epsilon}[L]$ of period $N$ such that
$\mu(e_\al)\in\<\pm 1\>e_{\mu(\al)}$ for $\al\in L$.
From \cite[Section 5]{Lep}, we can assume that
\begin{align}
  \mu(e_\al)=e_\al\quad\te{if }\mu(\al)=\al.
\end{align}
The following is straightforward (cf. Proposition \ref{prop:AQ-mod-hom-VA-hom}):

\begin{lem}\label{lem:mu-auto-VL}
There exists an automorphism $\wh{\mu}$ of $V_L$ which is uniquely determined by
\begin{align}
 \wh{\mu}(h)=\mu(h),\quad \quad
  \wh{\mu}(e_\al)=\mu(e_\al)\quad\te{for }h\in\h,\   \al\in L.
\end{align}
\end{lem}

Identify $\Z_{N}$ with $\<\wh{\mu}\>$ as an automorphism group of the vertex algebra $V_L$.
Set
\begin{align}
\xi=\exp(2\pi i/N),
\end{align}
the principal primitive $N$-th root of unity.

\begin{de}\label{def-epsilon-mu}
Define a map  $ \epsilon_\mu(\cdot,\cdot):\   \  L\times L\longrightarrow \C^\times$ by
\begin{align}\label{eq:2-cocycle-mu}
\epsilon_\mu (\al,\be)= \epsilon(\al,\be) \prod_{0\ne k\in\Z_N}(1-\xi^k)^{-\<\al,\mu^k(\be)\>}.
\end{align}
\end{de}

It is straightforward to show that $\epsilon_\mu$ is a 2-cocycle with
$\epsilon_\mu (\al,0)=1=\epsilon_\mu (0,\al)$.
Furthermore,   the commutator map of $\epsilon_{\mu}$ is given by
\begin{align}
c_{\mu}(\alpha,\beta)
=(-1)^{\<\alpha,\beta\>}\prod_{0\ne k\in \Z_N}(1-\xi^{k})^{\<\beta,\mu^k\alpha\>-\<\alpha,\mu^k\beta\>}
=\prod_{k\in\Z_N}(-\xi^{k})^{\<\mu^k\al,\be\>}=\prod_{k\in\Z_N}(-\xi^{k})^{-\<\al,\mu^k\be\>}.
\end{align}

Associated to the 2-cocycle $\epsilon_{\mu}$, we have the twisted group algebra $\C_{\epsilon_\mu}[L]$
which has a basis $\set{e_\al^\mu}{\al\in L}$ with
\begin{align*}
  e_\al^\mu\cdot e_\be^\mu=\epsilon_{\mu}(\al,\be)e_{\al+\be}^\mu\quad\mbox{ for }\al,\be\in L.
\end{align*}
Notice that $c_{\mu}(\mu(\al),\mu(\be))=c_{\mu}(\al,\be)$ for $\al,\be\in L$.
Then $\mu$ can also be lifted to an automorphism $\wh\mu$ on $\C_{\epsilon_\mu}[L]$
(see \cite[Section 5]{Lep}) such that for $\al\in L$,
\begin{align}
\wh\mu(e_\al^\mu)\in\<\xi'\>e_{\mu(\al)}^\mu\   \mbox{ and }   \  \wh\mu(e_\al^\mu)=e_\al^\mu\  \   \te{if }\mu(\al)=\al.
\end{align}

For the rest of this subsection, we assume
\begin{align}\label{p(x)-r}
p(x)=x^{r+1}p_{0}(x^N),
\end{align}
where $r\in \Z$ and $p_0(x)\in \C[x]$ with $p_0(0)\ne 0$.
Set $\phi(x,z)=\exp(zp(x)\partial_x)x$.


Recall (cf. \cite[Section 6]{LL}) that $V_{L}$ is a conformal vertex algebra, which is $\Z$-graded by
the conformal weight. Especially, we have
\begin{align}
  &z^{L(0)}Y(u,x)z^{-L(0)}=Y(z^{L(0)}u,zx),\\
  &L(0)h=h,\quad L(0)e_\al=\half\<\al,\al\>e_\al
\end{align}
for $u\in V_L$, $h\in\h,$ $\al\in L$. On the other hand, $\wh\mu$ preserves the conformal vector,
so that $[L(0),\wh\mu]=0$.
Set
\begin{align}
  L_\phi=\xi^{-rL(0)},
  \end{align}
 recalling that $r$ is the integer in (\ref{p(x)-r}).  Then we have
\begin{align}
  L_\phi Y(u,x)v=Y\left(L_\phi u, \xi^{-r}x\right)L_\phi v\   \    \   \mbox{ for }u,v\in V_{L}.
\end{align}

We now fix a  linear character
\begin{align}
\chi_{\phi}:\   \Z_N\rightarrow \C^\times  \  \mbox{ with }\chi_{\phi}([k])= \xi^{k}.
\end{align}
Set $\chi=\chi_{\phi}^{-r}$, another linear character of $\Z_{N}$. More specifically,
we have $\chi([k])=\xi^{-rk}$ for $[k]\in \Z_{N}$.
Furthermore,  we define a  group homomorphism
$$R:\   \Z_N\rightarrow \te{GL}(V_L)$$
 by
\begin{align}
R(k)=\wh\mu^{k}\chi(k)^{L(0)}=\wh\mu^{k}\xi^{-krL(0)}.
\end{align}
Then $(V_L,R)$ is a $(\Z_N,\chi)$-module vertex algebra.

In the following, we are going to construct $(\Z_N,\chi_{\phi})$-equivariant $\phi$-coordinated quasi $V_L$-modules.
For $h\in\h,\  n\in\Z$, set
\begin{align}
  h_{(n)}=\sum_{k=1}^{N}\mu^k(h)\xi^{-nk}\in \h.
\end{align}
We have
\begin{align}
\mu h_{(n)}=\xi^{n}h_{(n)}\   \mbox{ and }\  h=\frac{1}{N}(h_{(1)}+\cdots +h_{(N)}).
\end{align}
Recall that for $n\in \Z$, we have $\sum_{k=1}^{N}\xi^{nk}=0$ if $n\notin N\Z$.
Set
\begin{align}
\h\uz=\{h_{(0)}\ |\ h\in \h\},\    \   \   \   L_{(0)}=\{ \alpha_{(0)}\ |\  \alpha\in L\},
\end{align}
where $\h\uz$ is a subspace of $\h$ and $L\uz$ is a subgroup of $L$.

\begin{de}\label{def-T}
Let $\mathcal T$ denote the category of $L\uz$-graded $\C_{\epsilon_\mu}[L]$-modules
$T=\oplus_{\gamma\in L\uz}T_\gamma$ such that
\begin{align}\label{def-T-modules}
  e_\al^{\mu} w\in T_{(\al+\be)\uz},\quad\wh\mu(e_\al^{\mu})w=
  \xi^{-\<\al\uz,\be\>-\<\al\uz,\al\>/2}e_\al^{\mu} w
\end{align}
for $\al,\be\in L,\ w\in T_{\be\uz}$.
\end{de}

Note that all irreducible objects in category $\mathcal T$ were classified and constructed explicitly
in \cite[Section 6]{Lep}.

\begin{de}\label{def-tildemu}
We denote also by $\mu$  the automorphism of Lie algebra $\wh\h$ defined by
\begin{align}\label{mu-auto-wh-h}
\mu ({\bf k})={\bf k},\     \  \mu(h(n))=\xi^{-n}\mu(h)(n)\   \   \mbox{ for }h\in \h,\ n\in \Z.
\end{align}
\end{de}

Denote  by $\wh\h^\mu$ the fixed-point subalgebra of $\wh\h$ under
$\mu$.  For $h\in \h,\ n\in \Z$,  set
\begin{align}
h^\mu(n)=h_{(n)}(n)=\sum_{k\in \Z_{N}}(\mu^kh)(n)\xi^{-nk}\in \wh\h.
\end{align}
Then
\begin{align}
\wh\h^\mu=\te{span}\set{h^{\mu}(n)}{h\in\h,\  n\in\Z}\oplus \C{\bf k},
\end{align}
where
\begin{align}
[a^{\mu}(m),b^{\mu}(n)]=m\delta_{m+n,0}\<a_{(m)},b\>(N{\bf k})
\end{align}
for $a,b\in \h,\ m,n\in \Z$.

Set
\begin{align}
(\wh\h^\mu)'=\te{span}\set{h_{\mu}(n)}{h\in\h,\  n\in\Z,\  n\ne 0}\oplus \C{\bf k},
\end{align}
which is a Heisenberg algebra. Identify $\h\uz$ as a subspace of $\wh\h^\mu$ in the obvious way.
Then $\wh\h_\mu=(\wh\h^\mu)'\oplus \h\uz$, a direct sum of Lie algebras.
Set
$$\wh\h^\mu_-={\rm span} \{ h^\mu(-n)\ |\ h\in \h,\ n\ge 1\}, $$
an abelian subalgebra of $(\wh\h^\mu)'$.
It was well known that there is an irreducible $(\wh\h^\mu)'$-module structure of any nonzero level on
$S(\wh\h^\mu_-)$. We here view $S(\wh\h^\mu_-)$
as an $\wh\h^\mu$-module of {\em level $1/N$} with $\h\uz$ acting trivially.

Let $T$ be an $L\uz$-graded $\C_{\epsilon_\mu}[L]$-module from the category $\mathcal T$.
We make $T$ an $\wh\h^\mu$-module by letting $(\wh\h^\mu)'$ act trivially
and letting $\h_{(0)}$ act by
\begin{align*}
  h_{(0)}w=\<h\uz,\be\>w\quad\te{for }h\in \h,\  w\in T_{(\be\uz)},\,\be\in L.
\end{align*}
Set
\begin{align}
  V_T=S(\wh\h^\mu_-)\ot T,
\end{align}
the tensor product of $\wh\h^\mu$-modules, which is  of level $1/N$.

For $\al\in L$,  set
\begin{align}
E^\mu_\pm(\al,x)=\exp\(\sum_{n\in \Z_{\pm}}\al^\mu(n)\frac{x^{-n}}{n}\),
\end{align}
where $\Z_{\pm}$ denote the sets of positive (negative) integers.
The following results can be found in \cite{Lep}:

\begin{lem}\label{Lep-Epm}
The following relations hold on $V_{T}$ for $\al,\beta\in L$:
\begin{align}
E^\mu_\pm(\mu\al,x)=E^\mu_\pm(\al,\xi^{-1}x),
\end{align}
\begin{align}\label{Epm-comm}
E^\mu_{+}(\al,x_1)E^\mu_{-}(\beta,x_2)
=\left(\prod_{k\in \Z_{N}}\left(1-\frac{\xi^{-k}x_2}{x_1}\right)^{\<\mu^k\alpha,\beta\>}\right)
E^\mu_{-}(\beta,x_2)E^\mu_{+}(\al,x_1).
\end{align}
\end{lem}

\begin{de}\label{dYT0}
Let $T$ be given as before. Define a linear map
$$Y_T^0(\cdot,x): \  \h\oplus \C_{\epsilon}[L]\rightarrow (\te{End} V_T)[[x,x^{-1}]]$$
by
\begin{align}
&  Y_T^0(h,x)=p(x)\sum_{m\in\Z}h^{\mu}(m)x^{-m-1}=p(x)\sum_{k\in \Z_{N}}\xi^{k}\mu^{k}(h)(\xi^kx),\\
&  Y_T^0(e_\al,x)=E^\mu_-(-\al,x)E^\mu_+(-\al,x)e_\al x^{\al\uz}p(x)^{\half\<\al,\al\>}
x^{\frac{1}{2}\<\alpha_{(0)}-\al,\al\>}
\end{align}
for $h\in\h,\  \al\in L$.
\end{de}

Note that  for $\al\in L$, we have
$\al_{(0)}=\sum_{k\in \Z_{N}}\mu^k\al$.
Since $$\frac{1}{2}\<\alpha_{(0)},\al\>=\sum_{k\in\Z_N}\half\<\mu^k\al,\al\> \in \Z$$
 by (\ref{sum-even}) and since $\half \<\alpha,\alpha\>\in \Z$, we have
$\frac{1}{2}\<\alpha_{(0)}-\al,\al\> \in \Z$.

\begin{rem}\label{simple-facts}
{\em We here give some simple facts. For $\al\in L$,  we have
\begin{align}\label{fact1}
\sum_{0\ne k\in\Z_N}\<\mu^k(\al),\al\>=\sum_{k\in\Z_N}\<\mu^k(\al),\al\>-\<\al,\al\>=\<\al_{(0)}-\al,\al\>,
\end{align}
\begin{align}\label{fact2}
 & \sum_{0\ne k\in\Z_N}\frac{1}{1-\xi^k}\<\mu^k(\al),\al\>\nonumber\\
 =& \half \(\sum_{0\ne k\in\Z_N}\frac{1}{1-\xi^k}\<\mu^k(\al),\al\>
  + \sum_{0\ne k\in\Z_N}\frac{1}{1-\xi^{-k}}\<\mu^{-k}(\al),\al\>\) \nonumber\\
  =&\half \sum_{0\ne k\in\Z_N}\<\mu^k(\al),\al\>\(\frac{1}{1-\xi^k}
  +\frac{1}{1-\xi^{-k}}\) \nonumber\\
  =&\half\sum_{0\ne k\in\Z_N}\<\mu^k(\al),\al\>\nonumber\\
  =&\half\<\al_{(0)}-\al,\al\>,
\end{align}
\begin{align}\label{fact3}
 \sum_{0\ne k\in\Z_N}\frac{\xi^k}{1-\xi^k}\<\mu^k(\al),\al\>
 = \sum_{0\ne k\in\Z_N}\(\frac{1}{1-\xi^k}-1\)\<\mu^k(\al),\al\>
 =\half \<\al-\al\uz,\al\>.\   \   \
\end{align}}
\end{rem}

We have:

\begin{prop}\label{relations-AL}
The following relations hold on $V_T$ for $h,h'\in\h,\  \al,\be\in L$:
\begin{align}
&[Y_T^0(h,x_1),Y_T^0(h',x_2)]=\sum_{k\in\Z_N}\<\mu^k(h),h'\>(p(x_2)\partial_{x_2})
p(x_1)x_1\inverse\delta\(\xi^{-k}\frac{x_2}{x_1}\),\label{h-h'}\\
&[Y_T^0(h,x_1),Y_T^0(e_\al,x_2)]=\sum_{k\in\Z_N}\<\mu^k(h),\al\>Y_T^0(e_\al,x_2)
p(x_1)x_1\inverse\delta\(\xi^{-k}\frac{x_2}{x_1}\),\label{second-he-alpha}\\
&\prod_{0\ne k\in\Z_N}\(x_1/x_2-\xi^k\)^{-\<\al,\mu^k(\be)\>}
\(\frac{x_1-x_2}{p(x_2)}\)^{-\<\al,\be\>-1}    Y_T^0(e_\al,x_1)Y_T^0(e_\be,x_2)\nonumber\\
    &\  \   -
\prod_{0\ne k\in\Z_N}\(-\xi^k+x_1/x_2\)^{-\<\al,\mu^k(\be)\>}
    \(\frac{-x_2+x_1}{p(x_2)}\)^{-\<\al,\be\>-1}
    Y_T^0(e_\be,x_2)Y_T^0(e_\al,x_1)\nonumber\\
    &=\epsilon_\mu(\al,\be)Y_T^0(e_{\al+\be},x_2)
        p(x_1)x_1\inverse\delta\(\frac{x_2}{x_1}\).\label{eq:e-al-e-be-commutator}
\end{align}
\end{prop}

\begin{proof} The first relation holds as
\begin{align*}
&[Y_T^0(h,x_1),Y_T^0(h',x_2)]\nonumber\\
=&p(x_1)p(x_2)\sum_{\bar{r},\bar{s}\in \Z_{N}}\frac{1}{N}\<\mu^{\bar{r}}(h),\mu^{\bar{s}}(h')\>
\frac{\partial}{\partial x_2}x_1^{-1}\delta\(\xi^{\bar{s}-\bar{r}}\frac{x_2}{x_1}\)
\nonumber\\
=&p(x_1)p(x_2)\sum_{\bar{r},\bar{s}\in \Z_{N}}\frac{1}{N}\<\mu^{\bar{r}-\bar{s}}(h),h'\>
\frac{\partial}{\partial x_2}x_1^{-1}\delta\(\xi^{\bar{s}-\bar{r}}\frac{x_2}{x_1}\)
\nonumber\\
=&\sum_{k\in\Z_N}\<\mu^k(h),h'\>(p(x_2)\partial_{x_2})
p(x_1)x_1\inverse\delta\(\xi^{-k}\frac{x_2}{x_1}\).
\end{align*}
For $h\in\h$, set
\begin{align}
&h^\mu_\pm(x)=\sum_{n\in \Z_{\pm}}h^\mu(n)x^{-n-1}
=\sum_{n\in \Z_{\pm}}\sum_{k\in\Z_N}\xi^{-nk}(\mu^{k}h)(n)x^{-n-1},\\
&h^\mu(0)=\sum_{k\in \Z_{N}}(\mu^{k}h)(0).
\end{align}
We have
\begin{align*}
&\left[h^\mu_\pm(x_1),\sum_{n\in \Z_{\mp}}\al^\mu(n)\frac{x_2^{-n}}{n}\right]\nonumber\\
=&-\sum_{m\in \Z_{\pm}}\sum_{r,s\in \Z_{N}}\frac{1}{N}\<\mu^{r}h,\mu^s\al\>\xi^{m(s-r)}x_1^{-m-1}x_2^m
\nonumber\\
=&-\sum_{m\in \Z_{\pm}}\sum_{r,s\in \Z_{N}}\frac{1}{N}\<\mu^{r-s}h,\al\>\xi^{m(s-r)}x_1^{-m-1}x_2^m\nonumber\\
=&-\sum_{m\in \Z_{\pm}}\sum_{k\in \Z_{N}}\<\mu^{k}h,\al\>\xi^{-mk}x_1^{-m-1}x_2^m\nonumber\\
=&-\sum_{k\in \Z_{N}}\<\mu^{k}h,\al\>\frac{\xi^kx_1\inverse}{(x_1/x_2)^{\pm 1}-\xi^k}.
\end{align*}
Thus
\begin{align*}
 & [h_+^\mu(x_1),E_-^\mu(-\al,x_2)]=E_-^\mu(-\al,x_2)\sum_{k\in\Z_N}\<\mu^k(h),\al\>
  \frac{\xi^kx_1\inverse}{x_1/x_2-\xi^k},\\
  & [h_{-}^\mu(x_1),E_+^\mu(-\al,x_2)]=E_+^\mu(-\al,x_2)\sum_{k\in\Z_N}\<\mu^k(h),\al\>
  \frac{\xi^kx_1\inverse}{x_2/x_1-\xi^k}.
\end{align*}
We also have
\begin{align*}
\left[x_1^{-1}p(x_1)\sum_{k\in \Z_{N}}\mu^{k}(h)(0),e_{\alpha}\right]
=x_1^{-1}p(x_1)\sum_{k\in \Z_{N}}\<\mu^{k}(h),\al\>e_{\alpha}.
\end{align*}
Then we get the second relation by using the identity
$$\frac{\xi^k}{x_1/x_2-\xi^k}+\frac{\xi^k}{x_2/x_1-\xi^k}+1=\delta\left(\frac{\xi^{-k}x_2}{x_1}\right).$$
As for the last relation, we shall use (\ref{Epm-comm}).
Note that
\begin{align}
\prod_{k\in \Z_{N}}\left(\frac{x_1}{x_2}-\xi^{k}\right)^{-\<\alpha,\mu^k\beta\>}
=&\prod_{k\in \Z_{N}}
\left(\frac{x_1}{x_2}-\xi^{-k}\right)^{-\<\mu^k\alpha,\beta\>}   \nonumber\\
=&\prod_{k\in \Z_{N}}\left(1-\frac{\xi^{-k}x_2}{x_1}\right)^{-\<\mu^k\alpha,\beta\>}
\prod_{k\in \Z_{N}}\left(\frac{x_2}{x_1}\right)^{\<\mu^k\alpha,\beta\>}\nonumber\\
=&\left(\frac{x_2}{x_1}\right)^{\<\al\uz,\beta\>}\prod_{k\in \Z_{N}}\left(1-\frac{\xi^{-k}x_2}{x_1}\right)^{-\<\mu^k\alpha,\beta\>}
\end{align}
and that $x_1^{\al\uz}e_{\beta}^{\mu}=x_1^{\<\al\uz,\beta\>}e_{\be}^{\mu}x_1^{\al\uz}$.
Using these and (\ref{Epm-comm}) we get
\begin{align*}
&\prod_{0\ne k\in\Z_N}\(x_1/x_2-\xi^k\)^{-\<\al,\mu^k(\be)\>}
\(\frac{x_1-x_2}{p(x_2)}\)^{-\<\al,\be\>-1}    Y_T^0(e_\al,x_1)Y_T^0(e_\be,x_2)\nonumber\\
=&\prod_{k\in\Z_N}\(x_1/x_2-\xi^k\)^{-\<\al,\mu^k(\be)\>}
x_2^{-\<\al,\be\>}  p(x_2)^{\<\al,\be\>+1}  Y_T^0(e_\al,x_1)Y_T^0(e_\be,x_2)(x_1-x_2)^{-1}\nonumber\\
=&\epsilon_\mu(\al,\be)E_{-}^{\mu}(-\al,x_1)E_{-}^{\mu}(-\beta,x_2)
E_{+}^{\mu}(-\al,x_1)E_{+}^{\mu}(-\beta,x_2)e_{\al+\beta}^{\mu}x_1^{\al\uz}x_2^{\beta\uz}
x_2^{\<\al\uz-\al,\beta\>}\nonumber\\
&\times  x_1^{\frac{1}{2}\<\alpha_{(0)}-\al,\al\>}x_2^{\frac{1}{2}\<\beta_{(0)}-\beta,\beta\>}
p(x_1)^{\half\<\al,\al\>}p(x_2)^{\half\<\beta,\beta\>}p(x_2)^{\<\al,\beta\>+1}(x_1-x_2)^{-1}.
\end{align*}
Similarly, we have
\begin{align*}
&\prod_{0\ne k\in\Z_N}\(-\xi^k+x_1/x_2\)^{-\<\al,\mu^k(\be)\>}
    \(\frac{-x_2+x_1}{p(x_2)}\)^{-\<\al,\be\>-1}
    Y_T^0(e_\be,x_2)Y_T^0(e_\al,x_1)\nonumber\\
 =&   \prod_{k\in\Z_N}\(-\xi^k+x_1/x_2\)^{-\<\al,\mu^k(\be)\>}
    x_2^{-\<\al,\be\>}p(x_2)^{\<\al,\be\>+1}
    Y_T^0(e_\be,x_2)Y_T^0(e_\al,x_1)(-x_2+x_1)^{-1}\nonumber\\
    =&\epsilon_\mu(\be,\al)\prod_{k\in \Z_{N}}(-\xi^{k})^{-\<\al,\mu^k\be\>}E_{-}^{\mu}(-\al,x_1)E_{-}^{\mu}(-\beta,x_2)
E_{+}^{\mu}(-\al,x_1)E_{+}^{\mu}(-\beta,x_2)e_{\al+\beta}^{\mu}x_2^{\<\be\uz-\be,\al\>} \nonumber\\
&\times x_1^{\al\uz}x_2^{\beta\uz}
x_1^{\frac{1}{2}\<\alpha_{(0)}-\al,\al\>}x_2^{\frac{1}{2}\<\beta_{(0)}-\beta,\beta\>}p(x_1)^{\half\<\al,\al\>}
p(x_2)^{\half\<\beta,\beta\>}p(x_2)^{\<\al,\beta\>+1}(-x_2+x_1)^{-1}\nonumber\\
=&\epsilon_\mu(\al,\be)E_{-}^{\mu}(-\al,x_1)E_{-}^{\mu}(-\beta,x_2)
E_{+}^{\mu}(-\al,x_1)E_{+}^{\mu}(-\beta,x_2)e_{\al+\beta}^{\mu}x_2^{\<\be\uz-\be,\al\>} \nonumber\\
&\times x_1^{\al\uz}x_2^{\beta\uz}
x_1^{\frac{1}{2}\<\alpha_{(0)}-\al,\al\>}x_2^{\frac{1}{2}\<\beta_{(0)}-\beta,\beta\>}p(x_1)^{\half\<\al,\al\>}
p(x_2)^{\half\<\beta,\beta\>}p(x_2)^{\<\al,\beta\>+1}(-x_2+x_1)^{-1}.
\end{align*}
Recall that
$$\epsilon_\mu(\be,\al)\prod_{k\in \Z_{N}}(-\xi^{k})^{-\<\al,\mu^k\be\>}=\epsilon_\mu(\al,\be).$$
Then using the fact $(x_1-x_2)^{-1}-(-x_2+x_1)^{-1}=x_{1}^{-1}\delta(x_2/x_1)$
and the delta-function substitution we obtain the last relation.
\end{proof}

\begin{lem}\label{equivariance-generators}
The following relations hold on $V_{T}$ for $h\in \h,\ \al\in L$:
\begin{align*}
&Y_T^0(\mu (h),x)=\xi^{r}Y_T^0(h,\xi^{-1}x),\\
&Y_T^0(\wh\mu (e_{\al}),x)=\xi^{\frac{1}{2}\<\al,\al\>r}Y_T^0(e_{\al},\xi^{-1}x).
\end{align*}
\end{lem}

\begin{proof} Note that $p(\xi^{-1}x)=\xi^{-r-1}p(x)$. Using this we get
\begin{align*}
& Y_T^0(h,\xi^{-1}x)=p(\xi^{-1}x)\sum_{m\in\Z}h_{(m)}(m)(\xi^{-1}x)^{-m-1} \nonumber\\
=&\xi^{-r}p(x)\sum_{m\in\Z}\xi^{m}h_{(m)}(m)x^{-m-1}=\xi^{-r}Y_T^0(\mu (h),x),
\end{align*}
proving the first relation. For the second relation, we first verify that it is true on $T$.
 Let $w\in T_{(\beta_{(0)})}$ with $\beta\in L$. Recall the relation (\ref{def-T-modules}):
$$\wh{\mu}(e^{\mu}_{\al})w=\xi^{-\<\al_{(0)},\beta\>-\frac{1}{2}\<\al_{(0)},\al\>}e^{\mu}_\al w.$$
Then we have
\begin{align*}
&Y_{T}^{0}(e^{\mu}_{\al},\xi^{-1}x)w\nonumber\\
=&E^\mu_-(-\al,\xi^{-1}x)E^\mu_+(-\al,\xi^{-1}x)e_\al^{\mu} (\xi^{-1}x)^{\al\uz}p(\xi^{-1}x)^{\half\<\al,\al\>}
(\xi^{-1}x)^{\frac{1}{2}\<\alpha_{(0)}-\al,\al\>}w\nonumber\\
=&\xi^{\half\<\al,\al\>(-r)}\xi^{-\frac{1}{2}\<\alpha_{(0)},\al\>-\<\al_{(0)},\beta\>} E^\mu_-(-\mu(\al),x)E^\mu_+(-\mu(\al),x)
e_\al^{\mu} x^{\al\uz}p(x)^{\half\<\al,\al\>}
x^{\frac{1}{2}\<\alpha_{(0)}-\al,\al\>}w\nonumber\\
=&\xi^{\half\<\al,\al\>(-r)}E^\mu_-(-\mu(\al),x)E^\mu_+(-\mu(\al),x)
\hat{\mu}(e_\al^{\mu}) x^{(\mu\al)\uz}p(x)^{\half\<\mu\al,\mu\al\>}
x^{\frac{1}{2}\<(\mu\alpha)_{(0)}-\mu\al,\mu\al\>}w\nonumber\\
=&\xi^{\half\<\al,\al\>(-r)}Y_{T}^{0}(\wh{\mu}(e^{\mu}_{\al}),x)w,
\end{align*}
noticing that $\wh{\mu}(e_{\al})=\lambda e_{\mu(\al)}$ with $\lambda\in \C^{\times}$ and $(\mu\al)_{(0)}=\al\uz$.
This shows that the second relation holds on $T$.
Since $V_{T}$ as an $\wh\h^\mu$-module is generated by $T$,
it then follows from  (\ref{second-he-alpha}) in Proposition \ref{relations-AL} that the second relation holds on
the whole space $V_{T}$.
\end{proof}

Let $T$ be given as before. Set
$$U_T=\set{Y_T^0(u,x)}{u\in \h+ \C_{\epsilon}[L]}.$$
From the commutation relations in Proposition \ref{relations-AL} we see that
$U_{T}$ is a quasi $S$-local subspace of $\E(V_T)$, in particularly, $U_{T}$ is quasi compatible.
Then by Theorem \ref{thm:abs-construct-non-h-adic},
$U_T$ generates a nonlocal vertex algebra $\<U_T\>_{\phi}$ naturally
with $V_{T}$ as a $\phi$-coordinated quasi module.

Set
$$\Gamma_{N}=\{ \xi^{k}\ |\  k\in \Z\}\subset \C^{\times}.$$
From Proposition \ref{relations-AL},  we see that $U_T$ is $\Gamma_{N}$-quasi compatible.
On the other hand, from Lemma \ref{equivariance-generators}, $U_T$ is $\Gamma_{N}$-stable.
Fix a linear character $\chi_{\phi}: \Z_{N}\rightarrow \C^{\times}$ with $\chi_{\phi}(k)=\xi^{k}$, and
we define a group homomorphism $R:\Z_N\rightarrow \te{GL}(\E(V_T))$ by
\begin{align}
R(k)(a(x))=a(\xi^{-k}x)=a(\chi_{\phi}(k)^{-1}x)\   \   \   \mbox{ for }k\in\Z_N,\ a(x)\in \E(V_T).
\end{align}
Recall that $\chi: \Z_{N}\rightarrow \C^{\times}$ with $\chi=\chi_{\phi}^{-r}$.
With this, by Theorem \ref{coro:G-va-abs-construct-non-h-adic}  we immediately have:

\begin{coro}\label{lem:UT-va}
The nonlocal vertex algebra $\<U_T\>_\phi$ with the map $R$ is a $(\Z_N,\chi)$-module nonlocal vertex algebra and
$V_T$ is a $(\Z_N,\chi_{\phi})$-equivariant $\phi$-coordinated quasi
$\<U_T\>_\phi$-module with $Y_W(a(x),z)=a(z)$ for $a(x)\in \<U_T\>_\phi$.
\end{coro}

Furthermore, we have:

\begin{lem}\label{lem:lattice-va-quasi-phi-mod-YE-rel}
Let $h,h'\in\h$ and let $\al,\be\in L$. Then
\begin{align}
&[Y_\E^\phi(Y_T^0(h,x),x_1),Y_\E^\phi(Y_T^0(h',x),x_2)]
=\<h,h'\>\frac{\partial}{\partial x_2}x_1\inverse
\delta\(\frac{x_2}{x_1}\),\label{eq:lattice-va-phi-quasi-mod-rel-temp-1}\\
&[Y_\E^\phi(Y_T^0(h,x),x_1),Y_\E(Y_T^0(e_\al,x),x_2)]
=\<h,\al\>Y_\E^\phi(Y_T^0(e_\al,x),x_2)x_1\inverse
\delta\(\frac{x_2}{x_1}\),\label{eq:lattice-va-phi-quasi-mod-rel-temp-2}\\
&(x_1-x_2)^{-\<\al,\be\>-1}
    Y_\E^\phi(Y_T^0(e_\al,x),x_1)Y_\E^\phi(Y_T^0(e_\be,x),x_2)\nonumber\\
    &\  \  -
(-x_2+x_1)^{-\<\al,\be\>-1}
    Y_\E^\phi(Y_T^0(e_\be,x),x_2)Y_\E^\phi(Y_T^0(e_\al,x),x_1)\nonumber\\
    &=\epsilon(\al,\be)Y_\E^\phi(Y_T^0(e_{\al+\be},x),x_2)
    x_1\inverse\delta\(\frac{x_2}{x_1}\).\label{eq:lattice-va-phi-quasi-mod-rel-temp-3}
\end{align}
\end{lem}

\begin{proof}
Relations \eqref{eq:lattice-va-phi-quasi-mod-rel-temp-1} and \eqref{eq:lattice-va-phi-quasi-mod-rel-temp-2}  follow immediately from (\ref{h-h'}--\ref{second-he-alpha}) and Theorem \ref{prop:tech-calculation5}.
On the other hand, with relation (\ref{eq:e-al-e-be-commutator}), by Theorem \ref{prop:tech-calculation5} we have
\begin{align*}
  &\iota_{x,x_1,x_2}\prod_{0\ne k\in\Z_N}
  \(\phi(x,x_1)/\phi(x,x_2)-\xi^k\)^{-\<\al,\mu^k(\be)\>}
  \(\frac{\phi(x,x_1)-\phi(x,x_2)}{p(\phi(x,x_2))}\)^{-\<\al,\be\>-1}\\
&\qquad\times
    Y_\E^\phi(Y_T^0(e_\al,x),x_1)Y_\E^\phi(Y_T^0(e_\be,x),x_2)\nonumber\\
    -&\iota_{x,x_2,x_1}
\prod_{0\ne k\in\Z_N}
  \(\phi(x,x_1)/\phi(x,x_2)-\xi^k\)^{-\<\al,\mu^k(\be)\>}
  \(\frac{-\phi(x,x_2)+\phi(x,x_1)}{p(\phi(x,x_2))}\)^{-\<\al,\be\>-1}\\
&\qquad\times
    Y_\E^\phi(Y_T^0(e_\be,x),x_2)Y_\E^\phi(Y_T^0(e_\al,x),x_1)\\
=&\epsilon_\mu(\al,\be)Y_\E^\phi(Y_T^0(e_{\al+\be},x),x_2)
    x_1\inverse\delta\(\frac{x_2}{x_1}\).
\end{align*}
Noticing that
\begin{align*}
  \frac{\phi(x,x_1)-\phi(x,x_2)}{(x_1-x_2)p(\phi(x,x_2))},\quad \phi(x,x_1)/\phi(x,x_2)-\xi^k
\end{align*}
for $k\ne 0$ in $\Z_{N}$ are invertible elements of $\in \C((x))[[x_1,x_2]]$
and
\begin{align*}
  \lim\limits_{x_1\rightarrow x_2}\frac{\phi(x,x_1)-\phi(x,x_2)}{(x_1-x_2)p(\phi(x,x_2))}=1
  \quad\te{and}\quad
  \lim\limits_{x_1\rightarrow x_2}(\phi(x,x_1)/\phi(x,x_2)-\xi^k)=1-\xi^k,
\end{align*}
we have
\begin{align*}
 & \left(\frac{\phi(x,x_1)-\phi(x,x_2)}{(x_1-x_2)p(\phi(x,x_2))}\right)^{m}x_1\inverse\delta\(\frac{x_2}{x_1}\)
 =x_1\inverse\delta\(\frac{x_2}{x_1}\),\\
 &\left(\phi(x,x_1)/\phi(x,x_2)-\xi^k\right)^{m}x_1\inverse\delta\(\frac{x_2}{x_1}\)
 =(1-\xi^{k})^{m}x_1\inverse\delta\(\frac{x_2}{x_1}\)
\end{align*}
for $m\in \Z$.
Then we get
\begin{align*}
&(x_1-x_2)^{-\<\al,\be\>-1}
    Y_\E^\phi(Y_T^0(e_\al,x),x_1)Y_\E^\phi(Y_T^0(e_\be,x),x_2)\nonumber\\
    &\  \  -
(-x_2+x_1)^{-\<\al,\be\>-1}
    Y_\E^\phi(Y_T^0(e_\be,x),x_2)Y_\E^\phi(Y_T^0(e_\al,x),x_1)\nonumber\\
  =&\(\prod_{0\ne k\in\Z_N}(1-\xi^k)^{\<\al,\mu^k(\be)\>}\)
    \epsilon_\mu(\al,\be)Y_\E^\phi(Y_T^0(e_{\al+\be},x),x_2)
    x_1\inverse\delta\(\frac{x_2}{x_1}\)\nonumber\\
  =&\epsilon(\al,\be)Y_\E^\phi(Y_T^0(e_{\al+\be},x),x_2)
    x_1\inverse\delta\(\frac{x_2}{x_1}\),
\end{align*}
recalling the definition of $\epsilon_\mu$ (see \eqref{eq:2-cocycle-mu}). This
proves \eqref{eq:lattice-va-phi-quasi-mod-rel-temp-3}.
\end{proof}

We also have:

\begin{lem}\label{normal-order-derivative}
The following relation holds in $\<U_T\>_\phi$ for $\al\in L$:
\begin{align}
Y_T^0(\al,x)_{-1}^\phi Y_T^0(e_\al,x)=  p(x)\frac{d}{dx} Y_T^0(e_\al,x).
\end{align}
\end{lem}

\begin{proof}
Set $\al^\mu_\pm(x)=p(x)\sum_{n\in \Z_{\pm}}\al^\mu(n)x^{-n-1}$.
Notice that
\begin{align*}
  [\al_+^\mu(x_1),E_-^\mu(-\al,x_2)]=\(\sum_{k\in\Z_N}\<\al,\mu^k(\al)\>
  \frac{\xi^kp(x_1)x_1\inverse}{x_1/x_2-\xi^k}\)E_-^\mu(-\al,x_2)
\end{align*}
and $\prod_{k\in\Z_N}(x_1/x-\xi^k)=(x_1/x)^{N}-1$. Using these we have
\begin{eqnarray*}
 &&((x_1/x)^{N}-1)Y_T^0(\al,x_1)Y_T^0(e_\al,x)\\
  &=&((x_1/x)^{N}-1)\Bigg(\al^\mu_-(x_1)Y_T^0(e_\al,x)
  +Y_T^0(e_\al,x)\al^\mu_+(x_1)+Y_T^0(e_\al,x)\al^\mu(0)p(x_1)x_1\inverse\\
  &&+\<\al\uz,\al\>Y_T^0(e_\al,x)p(x_1)x_1\inverse
  +\sum_{k\in\Z_N}Y_T^0(e_\al,x)\<\al,\mu^k(\al)\>\frac{\xi^kp(x_1)x_1\inverse}{x_1/x-\xi^k}
  \Bigg),
\end{eqnarray*}
which lies in $\te{Hom}(V_T,V_T((x_1,x)))$.  Then
\begin{align*}
  &Y_\E^\phi(Y_T^0(\al,x),z)Y_T^0(e_\al,x)\\
  =&\left((\phi(x,z)/x)^{N}-1\right)^{-1}
  \left[((x_1/x)^{N}-1)Y_T^0(\al,x_1)Y_T^0(e_\al,x)\right]\Bigg|_{x_1=\phi(x,z)}\\
  =&\al^\mu_-(\phi(x,z))Y_T^0(e_\al,x)
  +Y_T^0(e_\al,x)\al^\mu_+(\phi(x,z))+Y_T^0(e_\al,x)\al^\mu(0)p(\phi(x,z))\phi(x,z)\inverse
  \\
  &+\<\al\uz,\al\>Y_T^0(e_\al,x)p(\phi(x,z))\phi(x,z)\inverse\nonumber\\
  &+\sum_{k\in\Z_N}Y_T^0(e_\al,x)\<\al,\mu^k(\al)\>
  \frac{\xi^kp(\phi(x,z))\phi(x,z)\inverse}{\phi(x,z)/x-\xi^k}.
\end{align*}
(With $\phi(x,z)$ an invertible element of $\C((x))((z))$, $\left((\phi(x,z)/x)^{N}-1\right)^{-1}$
 is viewed as an element of $\C((x))((z))$.)
Note that
\begin{align*}
  \te{Res}_zz\inverse \frac{\xi^kp(\phi(x,z))\phi(x,z)\inverse}{\phi(x,z)/x-\xi^k}=\left\{
  \begin{array}{ll}
  \frac{1}{2}p'(x)-p(x)x\inverse &\te{if }k=0,\\
  \frac{\xi^k}{1-\xi^k}p(x)x\inverse &\te{if }k\ne 0.
  \end{array}\right.
\end{align*}
Then we have
\begin{eqnarray*}
 & &Y_T^0(\al,x)_{-1}^\phi Y_T^0(e_\al,x)\nonumber\\
  &=&\te{Res}_zz\inverse Y_\E^\phi(Y_T^0(\al,x),z)Y_T^0(e_\al,x)\nonumber\\
  &=&\al^\mu_-(x)Y_T^0(e_\al,x)
  +Y_T^0(e_\al,x)\al^\mu_+(x)+Y_T^0(e_\al,x)\al^\mu(0)p(x)x\inverse\nonumber\\
  &&+Y_T^0(e_\al,x)\(\frac{1}{2}\<\al,\al\>p'(x)+\<\al\uz-\al,\al\>p(x)x^{-1}
  +\sum_{0\ne k\in\Z_N}\frac{\xi^k}{1-\xi^k}\<\mu^k(\al),\al\>p(x)x\inverse \).
\end{eqnarray*}
Using the simple facts in Remark \ref{simple-facts}, we obtain
\begin{align}
  &Y_T^0(\al,x)_{-1}^\phi Y_T^0(e_\al,x)\nonumber\\
  =&\te{Res}_zz\inverse Y_\E^\phi(Y_T^0(\al,x),z)Y_T^0(e_\al,x)\nonumber\\
  =&\al^\mu_-(x)Y_T^0(e_\al,x)
  +Y_T^0(e_\al,x)\al^\mu_+(x)+Y_T^0(e_\al,x)\al^\mu(0)p(x)x\inverse\nonumber\\
  &+\frac{1}{2}\<\al,\al\>p'(x)Y_T^0(e_\al,x)
  +\frac{1}{2}\<\al\uz-\al,\al\>Y_T^0(e_\al,x)p(x)x\inverse\nonumber\\
  =&p(x)\frac{d}{dx} Y_T^0(e_\al,x),\nonumber
\end{align}
as desired.
\end{proof}

Now, we are in a position to present our main result of this section.

\begin{thm}
Let $T$ be an $L\uz$-graded $\C_{\epsilon_\mu}[L]$-module from category $\mathcal T$.
Then there exists a $(\Z_N,\chi_{\phi})$-equivariant $\phi$-coordinated quasi $V_L$-module structure
$Y_T(\cdot,x)$ on $V_T$, which is uniquely determined by
\begin{align}
 Y_T(h,x)=Y_T^0(h,x),\quad Y_T(e_\al,x)=Y_T^0(e_\al,x)\quad\te{for }h\in\h,\  \al\in L.
\end{align}
Furthermore, if $T$ is an irreducible $L\uz$-graded $\C_{\epsilon_\mu}[L]$-module, then $V_{T}$ is an irreducible
$\phi$-coordinated quasi $V_{L}$-module.
\end{thm}

\begin{proof}  First of all, it follows from the commutation relations in
 Lemma \ref{lem:lattice-va-quasi-phi-mod-YE-rel} that $\<U_T\>_\phi$ is in fact a vertex algebra.
From Lemma \ref{lem:lattice-va-quasi-phi-mod-YE-rel},
we see that the relations (AL1-3) and \eqref{A(L)457} hold on $\<U_T\>_\phi$ with
\begin{align*}
  h[z]= Y_\E^\phi(Y_T^0(h,x),z),\quad
  e_\al[z]= Y_\E^\phi(Y_T^0(e_\al,x),z)\quad\te{for }h\in\h,\  \al\in L.
\end{align*}
Noticing that $p(x)\partial_x$ is the canonical derivation (the ${\mathcal{D}}$-operator) of the vertex algebra $\<U_T\>_\phi$,
using Lemma \ref{normal-order-derivative} we get
\begin{align}\label{eq:lattice-va-quasi-phi-mod-YE-AL6-rel}
  \partial_zY_\E^\phi(Y_T^0(e_\al,x),z)&=Y_\E^\phi(p(x)\partial_xY_T^0(e_\al,x),z)
  =Y_\E^\phi(Y_T^0(\al,x)_{-1}^\phi Y_T^0(e_\al,x),z)\nonumber\\
&=\nord Y_\E^\phi(Y_T^0(\al,x),z)Y_\E^\phi(Y_T^0(e_\al,x),z)\nord.
\end{align}
Note that $Y_T^0(e_0,x)=1_{V_T}$ the vacuum vector of $\<U_T\>_\phi$.
By Proposition \ref{prop:AQ-mod-hom-VA-hom}, there exists
 a vertex algebra homomorphism
$\psi$ from $V_L$ to $\<U_T\>_\phi$ such that
$$\psi(h)=Y_{T}^0(h,x),\   \   \psi(e_{\al})=Y_{T}^{0}(e_{\al},x)\   \   \  \mbox{ for }h\in \h,\ \al\in L.$$
As  $V_{T}$ is a $\phi$-coordinated quasi  $\<U_T\>_\phi$-module with $Y_{W}(a(x),z)=a(z)$
for $a(x)\in \<U_T\>_\phi$ (by Corollary \ref{lem:UT-va}),
it then follows  that $V_T$ is a $\phi$-coordinated quasi
$V_L$-module with $Y_T(u,z)=Y_W(\psi(u),z)$ for $u\in V_L$. For $h\in \h,\ \al\in L$, we have
\begin{align*}
&Y_{T}(h,z)=Y_{W}(\psi(h),z)=Y_{W}(Y_{T}^{0}(h,x),z)=Y_{T}^{0}(h,z),\nonumber\\
&Y_{T}(e_{\al},z)=Y_{W}(\psi(e_{\al}),z)=Y_{W}(Y_{T}^{0}(e_{\al},x),z)=Y_{T}^{0}(e_{\al},z).
\end{align*}
From Lemma \ref{equivariance-generators} we have
\begin{align}
Y_T^0(R(k)h,x)=Y_T^0(h,\chi_{\phi}(k)^{-1} x),\quad
Y_T^0(R(k)e_\al,x)=Y_T^0(e_\al,\chi_{\phi}(k)^{-1}x).
\end{align}
Then
\begin{align*}
Y_{T}(R(k)h,z)=Y_{T}(h,\chi_{\phi}(k)^{-1}z),\   \   \  Y_{T}(R(k)e_{\al},z)=Y_{T}(e_{\al},\chi_{\phi}(k)^{-1}z).
\end{align*}
As $\h+\C_{\epsilon}[L]$ generates $V_L$ as a vertex algebra, by Lemma \ref{lem:G-equiv-phi-mod-construct},
 $V_{T}$ is a $(\Z_{N},\chi_{\phi})$-equivariant $\phi$-coordinated quasi $V_{L}$-module.

For the furthermore statement, from the construction of $V_{T}$ we see that
any submodule of $V_T$ is stable under the actions of
 $U(\wh\h^\mu)$ and $\C_{\epsilon_{\mu}}[L]$.
Then it follows  that $V_T$ is irreducible if $T$ is irreducible.
\end{proof}


\begin{thebibliography}{HJKOS}

\bibitem[BLP]{blp}
C.-M. Bai, H.-S. Li, and Y.-F. Pei, $\phi_{\epsilon}$-Coordinated modules for vertex algebras,
{\em J. Algebra} {\bf 426} (2015), 211-242.

\bibitem[BK]{bk}
B. Bakalov and V. Kac, Field algebras, {\em Internat. Math. Res. Notices}
{\bf 3} (2003), 123-159.

\bibitem[Bo]{bor}
R. E. Borcherds, Vertex algebras, Kac-Moody algebras, and the Monster,
{\em Proc. Natl. Acad. Sci. USA} {\bf 83} (1986), 3068-3071.


\bibitem[D1]{dong1}
C. Dong, Vertex algebras associated with even lattices, {\em J. Algebra} {\bf 160} (1993), 245-265.

\bibitem[D2]{dong2}
C. Dong, Twisted modules for vertex algebras associated with even lattices, {\em J. Algebra} {\bf 165} (1994), 91-112.

\bibitem[DL1]{dl}
C. Dong and J. Lepowsky, {\em Generalized Vertex Algebras and
Relative Vertex Operators}, Progress in Math., Vol. {\bf 112},
Birkh\"auser, Boston, 1993.

\bibitem[DL2]{dl2}
C. Dong and J. Lepowsky, The algebraic structure of relative twisted vertex operators,
{\em J. Pure Applied Algebra} {\bf 110} (1996), 259-295.

\bibitem[DLM]{DLM}
C. Dong, H. Li and G. Mason, Regularity of rational vertex operator algebras,
{\em Advances in Math.} {\bf 132} (1997), 148-166.

\bibitem [EK]{ek}
P. Etingof and D. Kazhdan, Quantization of Lie bialgebras, V,
{\em Selecta Math. (N.S.)} {\bf 6} (2000), 105-130.



\bibitem[FHL]{fhl}
 I. B. Frenkel, Y.-Z. Huang, and J. Lepowsky, {\em On axiomatic approaches to
vertex operator algebras and modules}, Memoirs Amer. Math. Soc. 104, 1993.


\bibitem[FLM]{FLM}
I. Frenkel, J. Lepowsky and A. Meurman, {\it Vertex Operator
Algebras and the Monster}, Pure and Appl. Math., {\bf Vol. 134},
Academic Press, Boston, 1988.


\bibitem[H]{Ha-formal-group-book}
M. Hazewinkel, {\em Formal Groups and Applications,} Pure and Appl. Math., Vol. 78, London: Academic Press, 1978.

\bibitem[L]{Lep}
J. Lepowsky,  Calculus of twisted vertex operators, {\em Proc. Natl. Acad. Sci. USA} {\bf 82} (1985), 8295-8299.

\bibitem[LL]{LL}
J. Lepowsky and H.-S. Li, {\em Introduction to Vertex Operator
Algebras and Their Representations}, Progress in Math. {\bf 227},
Birkh\"auser, Boston, 2004.






\bibitem[Li1]{li-g1}
H.-S. Li, Axiomatic $G_{1}$-vertex algebras,
{\em Commun. Contemporary Math.} {\bf 5} (2003), 1-47.

\bibitem[Li2]{li-new}
H.-S. Li, A new construction of vertex algebras and quasi modules
for vertex algebras, {\em Advances in Math.} {\bf 202} (2006),
232-286.

\bibitem[Li3]{li-nonlocal}
H.-S. Li,  Nonlocal vertex algebras generated by formal vertex operators,
{\em Selecta Math. (N.S.)} {\bf 11} (2005), 349-397.

\bibitem[Li4]{li-qva2}
H.-S. Li, Constructing quantum vertex algebras,
{\em International Journal of Mathematics} {\bf 17} (2006), 441-476.

\bibitem[Li5]{li-gamma}
H.-S. Li, On certain generalizations of twisted affine Lie algebras
and quasi-modules for $\Gamma$-vertex algebras,
{\em J. Pure and Applied Algebra} {\bf 209} (2007), 853-871.

\bibitem[Li6]{li-twisted-quasi}
H.-S. Li, Twisted modules and quasi-modules for vertex operator algebras, in
the Proceedings of the International Conference in honor of
Professors James Lepowsky and Robert Wilson, Contemporary Math. {\bf
422} (2007), 389-400.

\bibitem[Li7]{li-qvta}
H.-S. Li, Quantum vertex ${\mathbb{F}}((t))$-algebras and their modules,
{\em J. Algebra} {\bf 324} (2010), 2262-2304.

\bibitem[Li8]{li-cmp}
H.-S. Li,  $\phi$-coordinated quasi-modules for quantum vertex algebras,
 {\em Commun. Math. Phys.} {\bf 308} (2011), 703-741.


\bibitem[Li9]{li-jmp}
H.-S. Li, $G$-equivariant $\phi$-coordinated quasi-modules for quantum vertex algebras,
 {\em J. Math. Phys.} {\bf 54} (2013), 1-26.

\bibitem[LTW]{ltw}
H.-S. Li, S. Tan and Qing Wang, Twisted modules for quantum vertex algebras,
 {\em J. Pure Applied Algebra} {\bf 214} (2010), 201-220.




\end{thebibliography}
\end{document}